\documentclass[10pt,a4paper]{amsart}

\usepackage[utf8]{inputenc}
\usepackage[T1]{fontenc}
\usepackage{amsmath}
\usepackage{amsfonts}
\usepackage{amssymb}

\usepackage{amsthm}
\usepackage{tikz}
\usetikzlibrary{cd}
\theoremstyle{plain}
\usepackage{enumitem}
\usepackage{mathtools}
\usepackage{stmaryrd}
\usepackage{hyperref}
\usepackage{cleveref}

\newtheorem{thm}{Theorem}[section]
\newtheorem{cor}[thm]{Corollary}
\newtheorem{lem}[thm]{Lemma}
\newtheorem{prop}[thm]{Proposition}
\theoremstyle{definition}
\newtheorem{rem}[thm]{Remark}
\newtheorem{definition}[thm]{Definition}
\newtheorem{ex}[thm]{Example}

\newtheorem{que}[thm]{Question}
\theoremstyle{plain}
\newtheorem{thmintro}{Theorem}

\newtheorem*{queintro*}{Question}

\numberwithin{equation}{section}

\newcommand{\Q}{\mathbb{Q}}

\newcommand{\R}{\mathbb{R}}
\newcommand{\C}{\mathbb{C}}
\newcommand{\Z}{\mathbb{Z}}

\newcommand{\N}{\mathbb{N}}
\newcommand{\CP}{\mathbb{CP}}

\newcommand{\cbba}{\operatorname{CBBA}}
\newcommand{\cdga}{\operatorname{cdga}}
\newcommand{\bico}{\operatorname{BiCo}}

\newcommand{\fg}{\mathfrak{g}}

\newcommand{\cB}{\mathcal{B}}

\newcommand{\cI}{\mathcal{I}}

\newcommand{\cL}{\mathcal{L}}

\newcommand{\im}{\operatorname{im}}
\newcommand{\del}{\partial}
\newcommand{\delbar}{{\bar{\partial}}}

\newcommand{\coker}{\operatorname{coker}}
\newcommand{\pr}{\operatorname{pr}}

\newcommand{\Id}{\operatorname{Id}}

\newcommand{\Ho}{\operatorname{Ho}}
\newcommand{\cone}{\operatorname{cone}}

\newcommand{\HAred}{\widetilde{H}_A}
\newcommand{\HBCred}{\widetilde{H}_{BC}}

\DeclareMathOperator{\Hom}{Hom}
\mathtoolsset{centercolon=true}
\newcommand{\Cdot}{{\raisebox{-0.7ex}[0pt][0pt]{\scalebox{2.0}{$\cdot$}}}}
\mathchardef\mhyphen="2D

\crefname{condition}{condition}{conditions}
\Crefname{condition}{Condition}{Conditions}

\let\oldabstract\abstract
\let\oldendabstract\endabstract
\makeatletter
\renewenvironment{abstract}
{%
	{\list{}{\addtolength{\leftmargin}{3em} 
			\listparindent 0em%
			\itemindent    \listparindent%
			\rightmargin   \leftmargin%
			\parsep        \z@ \@plus\p@}%
		\item\relax}%
	{\endlist}%
	\oldabstract}
{\oldendabstract}
\makeatother
\newcommand{\img}[2][1]{\begin{gathered}\includegraphics[scale=#1]{#2}\end{gathered}}
\newcommand{\Lpic}{\raisebox{-1pt}{$\img[0.75]{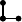}$}}
\newcommand{\hlinepic}{\raisebox{-1pt}{$\img[0.75]{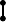}$}}
\newcommand{\linepic}{\img{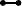}}
\newcommand{\revLpic}{\raisebox{-1pt}{$\img[0.75]{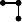}$}}
\title{Pluripotential homotopy theory}

\begin{document}
\author{Jonas Stelzig}
\begin{abstract}
	We build free, bigraded bidifferential algebra models for the forms on a complex manifold, with respect to a strong notion of quasi-isomorphism and compatible with the conjugation symmetry. This answers a question of Sullivan. The resulting theory naturally accomodates higher operations involving double primitives. As applications, we obtain various refinements of the homotopy groups, sensitive to the complex structure. Under a simple connectedness assumption, one obtains minimal models which are unique up to isomorphism and allow for explicit computations of the new invariants.
\end{abstract}
\maketitle

\section*{Introduction}

Consider the following vague
\begin{queintro*}
		What is the `holomorphic' extra structure on the homotopy type of a complex manifold?
\end{queintro*}

To explain what shape an answer could take, consider the case of cohomology:  The presence of a complex structure on a manifold $X$ equips the complex of $\C$-valued differential forms $A_X$ with a bigrading for which the differential $d=\del+\delbar$ decomposes into components of type $(1,0)$ and $(0,1)$. As a consequence of this simple observation there is a large collection of cohomology theories sensitive to the complex structure. For example, the de Rham cohomology is naturally a bifiltered vector space and one has the Dolbeault, Bott-Chern and Aeppli cohomologies \cite{Dol56}, \cite{BC65}, \cite{Aep62}
\[
H_{\delbar}:=\frac{\ker\delbar}{\im\delbar},\qquad H_{BC}:=\frac{\ker\del\cap\ker\delbar}{\im\del\delbar},\qquad H_A:=\frac{\ker \del\delbar}{\im\del+\im\delbar},
\] 
and many more (e.g. \cite{Fr55}, \cite{Var86}, \cite{Sch07}, \cite{PSU21}).

On the other hand, if one also takes into account the multiplicative structure, the graded-commutative differential graded algebra (cdga) of complex forms $A_X$, together with conjugation, represents the real homotopy type of a (say, simply-connected) manifold \cite{DGMS75}, \cite{Su77}, \cite{BG76}. To access the homotopy information, one has to choose a nilpotent model, i.e. a de Rham quasi-isomorphism $\Lambda V\to A_X$ from a cdga which is free as an algebra and satisfies a nilpotency condition recalled below (this amounts to a cofibrant replacement with respect to an appropriate model category structure). One is thus led to the following
\begin{queintro*}[Sullivan]
	Is there a free bigraded model for the cdga of forms on a complex manifold, invariant by conjugation?
\end{queintro*}

We will give a positive answer to Sullivan's question both in a conceptual and a computational sense. Before stating it, let us discuss quasi-isomorphisms. Given a map $f:\Lambda V\to A_X$ where both sides are bigraded and $f$ respects this structure, one may ask more of $f$ than inducing an isomorphism in total cohomology. For example, one may require $H_{\delbar}(f)$ to be an isomorphism, which implies $H_{dR}(f)$ to be an isomorphism if both sides are bounded, by a spectral sequence argument. We will use an even stronger notion that does not require boundedness assumptions: Namely, let us say $f$ is a bigraded, or pluripotential, quasi-isomorphism if the induced maps $H_{BC}(f)$ and $H_A(f)$ are isomorphisms. This quasi-isomorphism property is universal in a sense made precise in \Cref{thmintro: Quisos} below. In particular, the induced maps in row-, column and total cohomology are isomorphisms.

The conceptual answer to Sullivan's question then reads:
\begin{thmintro}\label{thmintro: Model Cat}
	The category of augmented, graded-commutative, bigraded, bidifferential algebras (cbba's) with a real structure carries a model category structure such that weak equivalences are bigraded quasi-isomorphisms and the cofibrant replacement of any cbba is nilpotent.
\end{thmintro}
Here, as in the singly graded case, a cbba is called nilpotent if it admits a presentation as a free bigraded algebra $\Lambda V$ with a well-ordered bigraded basis for $V$ such that the differential of every basis element lies in the subalgebra generated by the smaller basis elements.

Equipping the space of generators of a cofibrant replacement of $A_X$ augmented by evaluation at a point $x\in X$, with the linear part of the differentials, one can functorially associate to any pointed complex manifold a `homotopy bicomplex' $\pi^{\Cdot,\Cdot}(X,x)$. For simply connected (resp. nilpotent) spaces, the total cohomology of $\pi^{\Cdot,\Cdot}(X,x)$ recovers the dual of the complexified homotopy groups (and the dual of the Mal'cev completion of $\pi_1(X,x)$). Further, one obtains a homotopy version $\pi_\spadesuit$ of any other cohomological functor $H_{\spadesuit}$ (e.g. $H_{\delbar},H_{BC}$, Schweitzer cohomologies, ... or also universal diagrams involving these). In particular, one obtains a commutative diagram of maps and spectral sequences
\[
\begin{tikzcd}
	&\pi_{BC}^{p,q}(X,x)\ar[ld]\ar[d]\ar[rd]&\\
	\pi_{\delbar}^{p,q}(X,x)\ar[rd]\ar[r,Rightarrow]&(\pi_{p+q}(X,x)\ar[d]\otimes \C)^\vee &\pi_{\del}^{p,q}(X,x)\ar[ld]\ar[l, Rightarrow]\\
	&\pi_A^{p,q}(X,x)&
\end{tikzcd}
\]

The cofibrant replacements arising from \Cref{thmintro: Model Cat} make the functorial character of the theory evident, but they are very large. Just as in ordinary rational homotopy theory, for effective calculations, one wants `minimal' models. In the singly graded case, a nilpotent model is called minimal if $d$ has no linear part. Analogously, we call a cbba model bigradedly minimal if $\del\delbar$ has no linear part. The computational answer to Sullivan's question then reads:
\begin{thmintro}\label{thmintro: MiMo}
	For any connected compact complex manifold which is holomorphically simply connected, there exists a connected, real, degree-wise finite dimensional, bigradedly minimal model for the cbba of forms, which is unique up to isomorphism.
\end{thmintro}
Here, a connected compact complex manifold $X$ is called holomorphically simply connected, if ${H_A^1(X):=H_A^{1,0}(X)\oplus H_A^{0,1}(X)=0}$. This holds for example on compact K\"ahler manifolds with $b_1(X)=0$. A cbba $M$ is called connected if it is concentrated in non-negative total degree and satisfies $M^0=\C$. When a connected model exists, it is automatically an augmented model for any augmentation of the original algebra and so the choice of base-point becomes irrelevant (c.f. \Cref{Cor: independence of base-point}). In the non-simply connected case, connected, degree-wise finite-dimensional models may or may not exist and we give examples for both cases. The case of a compact K\"ahler manifold with the Hodge diamond of a complete intersection will be studied in detail, in particular we obtain a formality result (c.f. \Cref{thm: formality complete intersections}).

To prove the above results, we first revisit the additive theory, i.e. we study the (symmetric monoidal) category of bicomplexes which, as the category of (bigraded) representations of a Frobenius algebra, has a natural model category structure. The associated homotopy category $\Ho(\bico)$ can be naturally identified with the stable category of bicomplexes as in \cite{Hap88}, \cite{KQ20} and the notion of bigraded quasi-isomorphism introduced above appears naturally in this context, namely:

\begin{thmintro}\label{thmintro: Quisos}
	For a map $f:A\to B$ of bicomplexes the following statements are equivalent:
	\begin{enumerate}
		\item The map $f$ is an isomorphism in $\Ho(\bico)$.
		\item The maps $H_{BC}(f)$ and $H_A(f)$ are isomorphisms.
		\item For any additive functor $H$ from the category of bicomplexes to an additive category, vanishing on projective bicomplexes, $H(f)$ is an isomorphism.
	\end{enumerate}
If $A,B$ are bounded the above are further equivalent to:
\begin{enumerate}[resume]
	\item The maps $H_{\delbar}(f)$ and $H_{\del}(f)$ are isomorphisms.
\end{enumerate}
\end{thmintro} 
\Cref{thmintro: Quisos} recovers partial results in the bounded case obtained by elementary methods in \cite{Ste18} and puts them into a general algebraic context. It is a main technical input for Theorems \ref{thmintro: Model Cat} and \ref{thmintro: MiMo}, since $H_{BC}$ and $H_A$ are invariant by conjugation and using them to characterize quasi-isomorphisms one can argue -- in principle -- along the lines of the singly graded case, even though the technical details, in particular for \Cref{thmintro: MiMo}, are still considerable. As an outgrowth of our additive considerations, we prove several results that may be of independent interest: A formula which exhibits Dolbeault, Bott-Chern, Aeppli and, more generally, Schweitzer cohomology, as mapping spaces in $\Ho(\bico)$ giving rise to a long exact sequence in Schweitzer cohomology, and a K\"unneth type theorem for Aeppli- and Bott Chern cohomology (\Cref{cor: geometric ABC Kuenneth}).

The underlying idea of our main constructions is -- perhaps deceptively -- simple. At the most elementary level it consists in replacing $d$ by $\del\delbar$ in all the `right' places. In other words, one replaces ordinary potentials ($y=dx$) by pluripotentials ($y=\del\delbar x$). For example, in the singly graded theory, models are built by iteratively building pushouts of diagrams
\[
\begin{tikzcd}
\Lambda(\bullet)\ar[r]\ar[d]&\Lambda(\img{line})\\
M,&{}
\end{tikzcd}
\]
where $\bullet$ denotes a one-dimensional complex concentrated in a single degree and $\img{line}$ an isomorphism $d:\C\longrightarrow\C$. Depending on whether the vertical map sends the generator of $\bullet$ to an exact element or not, this corresponds to adding a cohomology class or enforcing a relation in cohomology. In perfect analogy, our constructions will use successive pushouts of the form
\[
\begin{tikzcd}
	\Lambda(\bullet)\ar[r]\ar[d]&\Lambda(\square)\\
	M,&{}
\end{tikzcd}
\]
where $\square$ denotes a four-dimensional bicomplex with $\del\delbar\neq 0$ (a square of four one-dimensional spaces in neighbouring bidegrees, connected by isomorphisms) and $\bullet\to \square$ is the inclusion identifying $\bullet$ with the top right corner, i.e. the sub-bicomplex given by $\im \del\delbar$.

\textbf{Related works.}
The development of homotopical versions of the cohomological story has been mostly focussed on the K\"ahler or algebraic case, see for instance \cite{Mor78}, \cite{Hai87a}, \cite{Hai87b}, \cite{NA87}. An exception is \cite{NT78}, where a homotopical analogue of Dolbeault cohomology and the Fr\"olicher spectral sequence is established for general complex manifolds, albeit at the cost of breaking the inherent conjugation symmetry, see also \cite{HT90}. In recent years, higher multiplicative operations involving `double primitives' (i.e. pluripotentials) have emerged \cite{AT15}, \cite{Ta17}, \cite{MS22}, which can be nontrivial even on manifolds satisfying the $\del\delbar$-property, \cite{ST22}, and which do not fit into the existing theories.

In the compact case, the theory developed here includes the Dolbeault cohomotopy and Fr\"olicher spectral sequence of \cite{NT78}, \cite{HT90} and the induced filtrations on the homotopy groups of compact K\"ahler manifolds obtained in \cite{DGMS75}, \cite{Mor78} via the principle of two types. It is also a natural framework for the pluripotential higher operations, a point which is elaborated in detail in \cite{MS22}.

\textbf{Sins of omission:} From a purely algebraic standpoint, one could discuss many results in a more general setup. For instance we always work over $\C$, but all proofs of results not concerning complex manifolds in sections \ref{sec: +}, resp. \ref{sec: wedge} carry over to arbitrary fields, resp. arbitrary fields of characteristic zero, and some can be adapted for modules over more general rings. One can also consider $n$-gradings for $n\geq 2$. Since our ultimate goal is to study complex manifolds we did not go down these roads.
For complex geometry itself, it will be very interesting to study the interaction of the bigraded invariants built here with the rational or integral structure. Furthermore, one may now ask which results and techniques of rational homotopy theory carry over to the holomorphic setting and which of the algebraic constructions given here have `geometric' counterparts, say in a suitable enlargement of the category of complex manifolds. We leave this for future work. Finally, instead of working with cbba's as we do here, there is also a complementary operadic approach using bigraded homotopy transfer methods. This will be explored in the forthcoming PhD thesis of Anna Sopena-Gilboy.

\textbf{Acknowledgements:} I am grateful to J. Cirici and the Universitat de Barcelona for an invitation to Barcelona and to Dennis Sullivan and the CUNY Graduate Center for an invitation to New York, both in early 2022, where I had the opportunity to present and discuss early versions of these results. Further, I thank S. Boucksom, J. Cirici, C. Deninger, D. Kotschick, A. Milivojevic, J. Morgan, D. Sullivan, S. Wilson, L. Zoller for useful questions, conversations or comments. Last but not least, I thank the anonymous referee for a careful reading and useful suggestions that improved the presentation.

\textbf{Notations and conventions:} For a bigraded object $V=V^{\Cdot,\Cdot}$, we write $V^\Cdot=\bigoplus_{p+q=\Cdot} V^{p,q}$ for the associated singly graded object using the total degree. For a pure (bi)degree element $v\in V$ of a bigraded object we use the notation $|v|$ to denote either total degree or bidegree, as will be clear from context, e.g. $|v|=(p,q)$ or $|v|=k$. We refer to \cite{Hov99} for background and definitions concerning model category structures. For any two objects $A,B$ of a model category, we denote by $[A,B]$ the set of morphisms in the homotopy category.\vspace{-1ex}
\newpage
\tableofcontents
\section{Direct sums and tensor products}\label{sec: +}
\subsection{The category of bicomplexes} 
\subsubsection{Basic definitions}
We denote by $\bico$ the category of bicomplexes (or double complexes) of $\C$-vector spaces. Objects are bigraded $\C$-vector spaces $A=\bigoplus_{p,q\in\Z} A^{p,q}$ together with an endomorphism $d$ of total degree $1$ which splits into components $d=\del+\delbar$ of degrees $(1,0)$ and $(0,1)$ and satsifies $d^2=0$. Morphisms are $\C$-linear maps of the underlying vector spaces that preserve the bigrading and commute with the differential. Even though we use the suggestive notation $\del,\delbar$, we do not in general assume bicomplexes to be equipped with a real structure in the sense introduced below.

The category $\bico$ has `shift' endofunctors $[r,s]$, defined by $(A[r,s])^{p,q}:=A^{p-r,q-s}$ and differential $d_{A[r,s]}=(-1)^{r+s}d$ and internal homs $\underline{\Hom}(A,B)$, defined by 
\[
\underline{\Hom}(A,B)^{p,q}:=\prod_{r,s\in\Z}\Hom_{\C-vs}(A^{p,q},B[r,s]^{p,q}),
\]
with differential 
\begin{equation}\label{eqn: diff hom cplx}
	d(\phi):=(a\mapsto [d,\varphi](a)=d\varphi(a)-(-1)^{|\varphi|}\varphi(d(a))).
\end{equation}We write $DA:=\underline{\Hom}(A,\C)$ for the dual complex and $D_nA:=DA[n,n]$. The tensor product 
\[(A\otimes B)^{p,q}=\bigoplus_{\substack{r+u=p\\s+v=q}} A^{r,s}\otimes_\C B^{u,v}.
\]
with differential $d(a\otimes b)=da\otimes b+(-1)^{|a|}a\otimes db$ equips $\bico$ with a closed symmetric monoidal structure. In particular, internal hom and tensor product are adjoint so that one has isomorphisms, natural in all entries
\begin{equation}\label{eqn: hom-tensor adjunction}
(\_\otimes A):\bico\leftrightarrows\bico:\underline{\Hom}(A,\_)
\end{equation}
For any finite dimensional bicomplex $A$, one has $A\cong DDA$ and there is an identification of functors $(DA\otimes\_)\cong \underline{\Hom}(A,\_)$. Thus, from \eqref{eqn: hom-tensor adjunction}, in this case $A\otimes \_$ has a left-adjoint:
\begin{equation}\label{eqn: tensor right adjoint}
(DA\otimes\_):\bico\leftrightarrows\bico:(\_\otimes A)
\end{equation}
\subsubsection{Real structures} We denote by $\R\bico$ the category of fixed points of the involution $\sigma$ on $\bico$ which maps a bicomplex $A$ to its complex conjugate bicomplex $\bar{A}$. The latter denotes the bicomplex which in degree $(p,q)$ is the space $A^{q,p}$ with the conjugate $\C$-vector space structure and for which $\del$ and $\delbar$ are interchanged. Objects of $\R\bico$ are pairs $(A,\sigma)$, where $\sigma:A\to \bar A$ is an isomorphism of bicomplexes. Equivalently, we may consider $\sigma$ as a $\C$-antilinear involution $A\to A$ such that $\sigma(A^{p,q})=A^{q,p}$ and $\sigma\del\sigma=\delbar$. An object of $\R\bico$ may also be considered as an ordinary (cochain) complex over the reals together with a bicomplex structure on its complexification, compatible with total grading and differential s.t. the conjugation on the coefficients acts as above. For every complex manifold $X$, the bicomplex of smooth, $\C$-valued differential forms $A_X$, with involution given by complex conjugation, is an object in $\R\bico$. The discussion of shifts, duals etc. has straightforward analogues for the category $\R\bico$. In fact, virtually all constructions made in this paper are either symmetric in $\del$ and $\delbar$ or are easily adapted to be and therefore extend to $\R\bico$.
\subsubsection{Bicomplexes as modules}
Following \cite{KQ20}, the category $\bico$ can also be described as the category of bigraded modules over the bigraded algebra $\Lambda(\del,\delbar)=H^\ast(S^1\times S^1;\C)$, where we write $\del,\delbar$ for the fundamental classes of the two factors and we consider the bigrading s.t. $|\del|=(1,0)$ and $|\delbar|=(0,1)$. As a bicomplex, this algebra looks as follows:
\[
\begin{tikzcd}
	\langle\delbar\rangle\ar[r]&\langle\del\delbar\rangle\\
	\C\ar[u]\ar[r]&\langle\del\rangle\ar[u]
\end{tikzcd}
\]
where all arrows are isomorphisms. We will thus write symbolically $\square:=\Lambda(\del,\delbar)$ if we consider this algebra as an object in $\bico$. Note that there is an antilinear conjugation action exchanging $\del$ and $\delbar$ and the category of equivariant modules for this action is $\R\bico$.
\subsubsection{Indecomposable bicomplexes} Recall that a bicomplex is called indecomposable if it cannot be decomposed into the direct sum of two nontrivial subcomplexes. The indecomposable bicomplexes fall in two classes: Squares, i.e. bicomplexes isomorphic to $\square[p,q]$, for $p,q\in \Z$ and zigzags, which can formally be defined as bicomplexes $Z$ such every nonzero component has dimension $1$ and the undirected support graph, with set of vertices $\{(p,q)\mid Z^{p,q}\neq 0\}$ and an edge between $(p,q)$ and $(r,s)$ whenever there is a nonzero differential $Z^{p,q}\to Z^{r,s}$, is homeomorphic to a non-empty interval (with or without either endpoint). Since $\del^2=\delbar^2=[\del,\delbar]=0$, any zigzag is concentrated in one or two total degrees. If a zigzag is bounded, it has finite dimension, equal to the number of nonzero components, and we call this dimension its length. Up to shift and isomorphism, the zigzags of length $1$ or $2$ look as follows, where the superscript denotes the bidegree and all maps are the identity:
\[
\C^{0,0},\quad
\begin{tikzcd}
	\C^{0,0}\ar[r]&\C^{1,0},
\end{tikzcd}
\quad
\begin{tikzcd}
	\C^{0,1}\\
	\C^{0,0}\ar[u]
\end{tikzcd},\quad
\begin{tikzcd}
	\C^{0,1}&\\
	\C^{0,0}\ar[r]\ar[u]&\C^{1,0},	
\end{tikzcd}
\begin{tikzcd}
	\C^{0,1}\ar[r]&	\C^{1,1}\\
	&	\C^{1,0}\ar[u].
\end{tikzcd}
\]
We will refer to these as dots and lines and (reverse) L's. Up to isomorphism and shift, there are two zigzags with any given length $\geq 2$.

\begin{thm}[\cite{KQ20},\cite{Ste18}]\label{thm:decomposition}
	Any indecomposable bicomplex is a square or a zigzag and any bicomplex is a direct sum of indecomposable subcomplexes. Such a decomposition is unique up to (noncanonical) isomorphism and ordering.
\end{thm}
Note that infinite dimensional zigzags will be unbounded in one or both anti-diagonal directions.
 \begin{definition}\label{def: locally bounded}
	A bicomplex is called \textbf{locally bounded} if it does not contain any infinite-length zigzags as direct summands.
\end{definition}
For example, first quadrant complexes, i.e. $A^{p,q}=0$ whenever $p<0$ or $q<0$, are automatically locally bounded. By the usual construction of the spectral sequence of a filtered complex applied to the row and column filtrations and a case by case inspection as in \cite[p.14]{KQ20}, we have
\begin{lem}\label{lem: locally bounded}
	For any bicomplex $A$ there are two `Fr\"olicher' spectral sequences with row, resp. column cohomology as their first pages. They both converge to the total cohomology if and only if $A$ is locally bounded.
\end{lem}
Let us call a bicomplex minimal if $\del\delbar \equiv 0$. E.g., zigzags are minimal, but squares are not. We note the following consequence of \Cref{thm:decomposition} for later use:
\begin{cor}[Minimal $\oplus$ Contractible decomposition]\label{cor: minimal + contractible}
	Any bicomplex $A$ admits a (noncanonical) decomposition $A=A^{zig}\oplus A^{sq}$, where $A^{sq}$ is a direct sum of squares and $A^{zig}$ satisfies $\del\delbar\equiv 0$.
\end{cor}

\subsection{The homotopy category of bicomplexes}
The algebra $\Lambda(\del,\delbar)$ is a Frobenius algebra and as such the category of modules over it is equipped with a canonical model category structure (c.f. \cite{Hov99}) and the associated homotopy category is a triangulated category as in \cite{Hap88}. Without claiming any originality, in this section we spell out these general results for the category of bicomplexes, following \cite{Hap88}, \cite{Hov99} and \cite{KQ20}.

Being the category of bigraded modules over a bigraded Frobenius algebra, projective objects in $\bico$ coincide with injective objects. In fact, projective objects are precisely the direct sum of squares. The same holds in $\R\bico$, indecomposable projective objects are direct sums of either conjugation invariant squares on the diagonal or pairs of conjugate squares.

Let us say a map $A\to B$ in $\bico$ is nullhomotopic if it factors through a projective object and two maps $f,g$ are homotopic $f\simeq g$ if their difference $f-g$ is nullhomotopic. We denote the space of nullhomotopic maps between $A,B$ by
\[
\Hom_{\bico}^0(A,B):=\{f\simeq 0\}\subseteq \Hom_{\bico}(A,B)
\] 
The following Lemma shows how this notion is a bigraded version of the familiar notion of chain homotopy, c.f. \cite{KQ20}.

\begin{lem}\label{lem: nullhomotopic morphisms}
	Let $f:A\to B$ be a map of bicomplexes. The following assertions are equivalent:
	\begin{enumerate}
		\item $f$ is nullhomotopic.
		\item There exists a map $h:A\to B[1,1]$ of bigraded vector spaces s.t. $f=[\del,[\delbar,h]]=\del\delbar h-\del h\delbar +\delbar h\del -h\delbar\del$.\footnote{This notion was independently suggested to the author by P. Deligne.}
	\end{enumerate}
\end{lem}

\begin{proof}
	Given any square $S$ with generator $x$, define a map $k:=k_S:S\to S[1,1]$ by $k(\del\delbar x)=x$ and $k(x)=k(\del x)=k(\delbar x)=0$. Denoting $\varphi_k:=[\del,[\delbar,k]]$, we have $\varphi_k=\Id_S$. This shows the statement is true for the identity map on a square, and hence on any projective object. Next, assume we have a factorisation \[f:A\overset{i}{\longrightarrow} P\overset{p}{\longrightarrow}B\] with $P$ projective and choose a map $k:P\to P[1,1]$ as above s.t. $\Id_P=[\del,[\delbar,k]]$. Then define $h:=p\circ k\circ i$. 
	
	Conversely, given $h:A\to B[1,1]$, set $P:=\square\otimes A[-1,-1]$, where we consider $\square=\Lambda(\del,\delbar)$ as a bicomplex as usual but consider $A[-1,-1]$ with $\del=\delbar=0$. With this definition, the inclusion $i:A\to P$ defined by $i(a)=\del\delbar\otimes a-\del\otimes \delbar a + \delbar \otimes \del a - 1\otimes \delbar\del a$ is a map of bicomplexes. Then, defining a map $p:P\to B$ by $p:=m\circ (\Id\otimes h)$ where $m:\Lambda(\del,\delbar)\otimes B\to B$ is the module structure, we have $f=p\circ i$.
\end{proof}
\begin{rem}
	If in the previous Lemma $f$ is a map in $\R\bico$, one may take $h$ to be purely imaginary (i.e. $\bar h=-h$).
\end{rem}
We say a map $f:A\to B$ is a bigraded chain homotopy equivalence if there exists a map $g:B\to A$ s.t. $f\circ g\simeq \Id_B$ and $g\circ f\simeq\Id_A$. We will give further characterisations of this property in \Cref{sec: bigraded quisos} below.

\begin{thm}
The category $\bico$ (resp. $\R\bico$) carries a cofibrantly generated model category structure, where 
\begin{enumerate}
\item Fibrations are the surjective maps of bicomplexes.
\item Weak equivalences are the bigraded chain homotopy equivalences.
\item Cofibrations are the injective maps of bicomplexes.
\end{enumerate}
\end{thm}

\begin{proof}
	This is (the bigraded version of) a special case of a general result on categories of modules over Frobenius algebras, see \cite{Hov99}.
\end{proof}

\begin{rem}
	In the literature one also finds model category structures on the category of bicomplexes that use other notions of weak equivalences, e.g. maps that induce isomorphisms in total cohomology or on some page of the spectral sequences c.f. \cite{MR18}, \cite{CESLW20}. As we see in \Cref{sec: bigraded quisos}, a bigraded chain homotopy equivalence is a weak equivalence in all these other structures.
\end{rem}
We denote by $\Ho(\bico)$ the homotopy category of $\bico$. Every object in $\bico$ is both fibrant and cofibrant, and so one may compute the morphisms as homotopy classes
\begin{equation}\label{eqn: homotopy classes in bico}
[A,B]=\frac{\Hom_{\bico}(A,B)}{\Hom_{\bico}^0(A,B)}.
\end{equation}
Again as a specialization of general results on stable categories 
\cite{Hap88}, \cite[Prop. 4.19]{Cis10}, $\Ho(\bico)$ can be equipped with the structure of a triangulated category as follows (c.f. also \cite{KQ20}):

A shift functor $\bico\to \bico$, which for reasons that will become obvious momentarily, we denote by $L$, is defined by $L(A):=\Lpic\otimes A$, where \[
\Lpic:=\square[-1,-1]/\langle\del\delbar\rangle=\begin{tikzcd}
\langle\delbar\rangle&\\
\C\ar[u]\ar[r]&\langle\del\rangle,
\end{tikzcd}
\]
with $\C$ sitting in degree $(-1,-1)$.
A homotopy-inverse to $L$ is given by $L^{-1}(A):=\revLpic\otimes A$, where
\[
\revLpic:=\ker(\square\to \C)=\begin{tikzcd}
\langle\delbar\rangle\ar[r]&\langle\del\delbar\rangle\\
&\langle\del\rangle\ar[u]
\end{tikzcd}
\]
\begin{lem}\label{lem: L^n}
In ${\Ho(\bico)}$, there are natural identifications $L\circ L^{-1}\cong \Id$. More generally, the action of $L^n$ is naturally identified with tensoring with a length $2|n|+1$ zigzag.
\end{lem}
\begin{proof}
The first part says the shift functor is an auto-equivalence, as is true in any triangulated category. The more general statement is essentially proved in \cite[§3]{Ste18}, though without giving an explicit isomorphism, which can be easily provided. In order not to get lost in notation, we give an example instead of a general calculation: Considering $\Lpic\otimes \Lpic$ and writing $x,y$ for generators in bottom left degree, we  find that the subcomplex $S\subseteq \Lpic\otimes\Lpic$ generated (as a bicomplex) by $x\otimes y$ is a square. Since injective-projective objects necessarily split off as direct summands, the projection to the quotient is thus an isomorphism in ${\Ho(\bico)}$ onto a length five zigzag. Concretely: 
	\[
	\begin{tikzcd}
	\langle \delbar x\rangle&\\
	\langle x\rangle\ar[r]\ar[u]&\langle\del x\rangle	
	\end{tikzcd}\otimes 
	\begin{tikzcd}
	\langle \delbar y\rangle&\\
	\langle y\rangle\ar[r]\ar[u]&\langle\del y\rangle	
\end{tikzcd}
\]
\[
=\begin{tikzcd}
\langle\delbar(xy)\rangle\ar[r]&\langle\del\delbar(xy)\rangle\\
\langle xy\rangle\ar[r]\ar[u]&\langle\del(xy)\rangle\ar[u]
\end{tikzcd}\oplus
\begin{tikzcd}
\langle\delbar x\delbar y\rangle&&\\
\langle x\delbar y\rangle\ar[r]\ar[u]&\langle\del x\delbar y\rangle&\\
&\langle \del x y\rangle \ar[r]\ar[u]&\langle \del x \del y\rangle
\end{tikzcd}
	\]	
\end{proof}

We may define a cone functor by setting \[
\cone(f:A\to B):=\coker(A\to B\oplus \square[-1,-1]\otimes A)
\] where the map is given by $a\mapsto (f(a),i\del\delbar\otimes a)$. Explicitly, as a bigraded vector space, we have an isomorphism $\cone(f)\cong B\oplus L(A)$, with the usual differential on $B$ and 
$d(1\otimes a)=d\otimes a+1\otimes da$, $d(\del\otimes a)=-f(ia)-\del\otimes da$ and $d(\delbar\otimes a)=f(ia)-\delbar\otimes da)$.
Finally, distinguished triangles are those isomorphic to
\[
A\overset{f}{\longrightarrow} B\longrightarrow \cone(f)\to L(A).
\]


The tensor product and internal homs make $\bico$ into a symmetric monoidal model category \cite[4.2.15]{Hov99}. In particular, they induce well-defined functors on the homotopy category and we still have adjunctions
\begin{equation}\label{eqn: tensor-hom adjunction stable cat}
(\_\otimes A):{\Ho(\bico)}\leftrightarrows{\Ho(\bico)}:\underline{\Hom}(A,\_)
\end{equation}
and thus, for any $A$ isomorphic in $\Ho(\bico)$ to a finite dimensional bicomplex:
\begin{equation}\label{eqn: tensor right adjoint stable cat}
(DA\otimes\_):{\Ho(\bico)}\leftrightarrows{\Ho(\bico)}:(\_\otimes A)
\end{equation}
and analogously for $\Ho(\R\bico)$.
\subsection{Aeppli, Bott-Chern and Schweitzer cohomology as mapping spaces}

Recall that we have functors from bicomplexes to bigraded vector spaces given by Bott-Chern and Aeppli cohomologies
\[
H_{BC}:=\frac{\ker\del\cap\ker\delbar}{\im \del\delbar}\qquad H_A:= \frac{\ker\del\delbar}{\im\del+\im\delbar}.
\]
If $A$ is a bicomplex, the bigrading $H_{BC}(A)=\bigoplus H_{BC}^{p,q}(A)$, $H_A(A)=\bigoplus H_A^{p,q}(A)$ is simply induced by that of $A$. 
\begin{ex}\label{ex: HBC and homotopy classes}
For bicomplexes $A,B$, there is a natural identification $[A,B]=H_{BC}^{0,0}(\underline{\Hom}(A,B))$. This follows from \Cref{eqn: homotopy classes in bico},  \Cref{lem: nullhomotopic morphisms} and \Cref{eqn: diff hom cplx}.
\end{ex}
The Bott-Chern and Aeppli cohomology groups are cohomology groups in particular degrees of a complex defined by M.~Schweitzer and J.~P.~ Demailly. More precisely, for any bicomplex $A$ and a pair of integers $p,q\in\Z$, define a complex $(\cL_{p,q}^\Cdot(A),d)$ via
\begin{align*}
	\cL_{p,q}^k(A)&:=\bigoplus_{\substack{r+s=k\\r<p,s<q}} A^{r,s}&\text{if }k\leq p+q-2,\\
	\cL_{p,q}^k(A)&:=\bigoplus_{\substack{r+s=k+1\\r\geq p,s\geq q}} A^{r,s}&\text{if }k\geq p+q-1,
\end{align*}
with differential given by
\[
\cdots\overset{\pr\circ d}{\longrightarrow}\cL_{p,q}^{p+q-3}(A)\overset{\pr\circ d}{\longrightarrow}\cL_{p,q}^{p+q-2}(A)\overset{\del\delbar}{\longrightarrow}\cL_{p,q}^{p+q-1}(A)\overset{d}{\longrightarrow}\cL_{p,q}^{p+q}(A)\overset{d}{\longrightarrow}\cdots
\]
Denoting its cohomology groups by $H_{S_{p,q}}^k(A):=H^k(\cL_{p,q}(A))$, one has $H_{BC}^{p,q}(A)=H^{p+q-1}_{S_{p,q}}(A)$ and $H_A^{p,q}(A)=H^{p+q}_{S_{p+1,q+1}}(A)$.

Write $\bullet$ for the double complex consisting of a single one-dimensional vector space $\C$ in degree $(0,0)$. Then:

\begin{thm}\label{BC A as hom sets}
	There are natural isomorphisms
	\[
[\bullet[p,q], A]\cong H^{p,q}_{BC}(A)\quad\text{and}\quad [\Lpic[p,q],A]\cong H_A^{p-1,q-1}(A)
	\]
	and, more generally for any $k\in\Z$,
	\[
	[\Lpic^{\otimes k}[p,q],A]\cong H_{S_{p,q}}^{p+q-k-1}(A).
	\]
\end{thm}
\begin{proof}
	We have \[
	[\bullet[p,q],A]=\Hom_{\bico}(\bullet[p,q],A)/\Hom_{\bico}^0(\bullet[p,q],A).
	\]
Let $x=1[p,q]$ be the generator for $\bullet[p,q]$. Then there is an isomorphism \begin{align*}
	\Hom_{\bico}(\bullet[p,q],A)&\cong\ker\del\cap\ker\delbar\cap A^{p,q}\\
	\varphi&\mapsto \varphi(x).\end{align*}
 On the other hand, by \Cref{lem: nullhomotopic morphisms}, the set of null-homotopic morphisms is given by the collection of $\varphi_h=[\del,[\delbar,h]]$ for $h:\bullet[p,q]\to A$ a linear map of bidegree $(-1,-1)$. In that case, $\varphi_h(x)=\del\delbar h(x)$, so that under the same identification $\Hom_{\bico}^0(\bullet[p,q],A)\cong\im\del\delbar$. This proves the first isomorphism. Note that we could have also argued via \Cref{ex: HBC and homotopy classes}.
	
	For the second, we again denote by $x=1[p-1,q-1]$ a bicomplex generator of $\Lpic[p,q]$, so that as a vector space $\Lpic[p,q]=\langle x,\del x,\delbar x\rangle$. As before, sending $\varphi\mapsto \varphi(x)$ gives an identification of $\Hom_{\bico}(\Lpic[p,q],A)\cong \ker\del\delbar\cap A^{p-1,q-1}$. On the other hand, to give a linear map $h:\Lpic[p,q]\to A$ of degree $(-1,-1)$ it is necessary and sufficient to give the three elements $h(x)\in A^{p-2,q-2}$, $h(\del x)\in A^{p-1,q-2}$ and $h(\delbar x)\in A^{p-2,q-1}$ and $\varphi_h(x)=\del\delbar h(x) - \del h(\delbar x) + \delbar h(\del x)$ so that $\Hom_{\bico}^0(\Lpic[p,q],A)=(\im\del+\im\delbar)\cap A^{p,q}$ 
	
For the general formula, use \Cref{lem: L^n} to replace $\Lpic^{\otimes k}$ by an appropriate zigzag of length $2|k|+1$. Then the same arguments as above, with more cumbersome notation, yield the conclusion. We omit the details.	
\end{proof}

\begin{cor}\label{cor: les Schweitzer}
	For any triangle $A\to B\to C\to LA$ in ${\Ho(\bico)}$, there is a long exact sequence
	\begin{multline}
	...\to H_{A}^{p-1,q-1}(C)\to H_{BC}^{p,q}(A)\to H_{BC}^{p,q}(B)\to H_{BC}^{p,q}(C)\\\to H_{S_{p,q}}^{p+q}(A)\to H_{S_{p,q}}^{p+q}(B)\to H_{S_{p,q}}^{p+q}(C)\to H_{S_{p,q}}^{p+q+1}(A)\to \cdots
	\end{multline}
\end{cor}

\begin{proof}
By general properties of triangulated categories, there is a long exact sequence
\begin{multline}
...\to [\bullet[p,q],L^{-1}(C)]\to [\bullet[p,q],A]\to [\bullet[p,q],B]\\\to [\bullet[p,q],C]\to [\bullet[p,q],L(A)]\to ...
\end{multline}
and this implies the statement using \Cref{BC A as hom sets} and that for all bicomplexes $E,F$ and integers $r\in\Z$ one has $[E,L^r(F)]=[L^{-r}(E),F]$ by \eqref{eqn: tensor right adjoint stable cat}.
\end{proof}
\begin{rem}\label{rem: Dolbeault les}
	Using the complex $\linepic:=(\C\to \C)$ concentrated in degrees $(0,0)$ and $(1,0)$, we see that $[\linepic[p,q],A]\cong H_{\delbar}^{p,q}(A)$ and we recover the long exact sequence in column cohomology. Considering longer even length zigzags $Z$ going from $(0,0)$ to $(r,-r+1)$, there is a surjection $[Z[p,q], A]\to E_r^{p,q}(A)$, which is in general not an isomorphism, as can be seen e.g. by taking $Z$ to be length $4$ and computing $[Z,\linepic[1,-1]]$.
\end{rem}
\begin{rem}
	An elementary construction of the long exact sequence in \Cref{cor: les Schweitzer} is as follows: Assume the triangle $A\to B\to C\to LA$ is associated to a short exact sequence $0\to A\to B\to C\to 0$ of bicomplexes. The sequence of simple complexes $0\to \cL_{p,q}(A)\to \cL_{p,q}(B)\to\cL_{p,q}(C)\to 0$ is still exact and one can consider its associated long exact sequence as usual.
\end{rem}
\begin{ex}
	For any map $f:X\to Y$ of compact complex manifolds, one obtains a triangle $A_Y\to A_X\to \cone(f)\to LA_Y$ in ${\Ho(\bico)}$ and a corresponding long exact sequence in Schweitzer cohomology and Dolbeault cohomology. E.g. one may apply this to submersions $f:X\to Y$ or embeddings $i:Z\to X$, in which case the cone is (up to shift) quasi-isomorphic to the relative forms $A_{X/Y}=\coker f^*$ or $A(X,Z):=\ker i^*$.
\end{ex}


To further illustrate the mapping-space long exact sequence in an abstract setting, we apply it for a computation that will be used later on, in \Cref{sec: Bigraded Minimal Models: Existence}:

Let $M$ be a double complex and $[c]\in H_{BC}^{p+1,q+1}(M)$ for some $p,q\in\Z$. Define $M\oplus_c \square[p,q]$ to be the pushout of the diagram
\[\begin{tikzcd}
	\bullet[p+1,q+1]\ar[r]\ar[d]&\square[p,q]\\
	M,
\end{tikzcd}
\] where the horizontal map sends the generator to the top right tip and the vertical map sends the generator to $c$. 
\begin{lem}[Bott-Chern-classes squared away I]\label{lem: ABC cohomology of squared away BC-class}
	Write $j=p+q$. 
	The cohomology $M\oplus_c \square[p,q]$ may be computed as follows:
	\begin{enumerate}
		\item If $[c]=0\in H_{BC}(M)$, we have $M\oplus_c\square[p,q]\cong M\oplus \Lpic[p+1,q+1]$ and so
			\begin{align*}
			H_{BC}^{i}(M\oplus_c \square[p,q])&\cong\begin{cases}
				H_{BC}^i(M)\oplus \C^2&\text{if }i=j+1\\
				H_{BC}^i(M)&\text{else.}
			\end{cases}\\
			H_{A}^{i}(M\oplus_c \square[p,q])&\cong\begin{cases}
				H_{A}^{i}(M)\oplus\C&\text{if } i=j\\
				H_A^i(M)&\text{else.}
			\end{cases}
		\end{align*}
		
		\item If $[c]\neq 0\in H_{BC}(M)$, we have
			\begin{align*}
			H_{BC}^{i}(M\oplus_c \square[p,q])&\cong\begin{cases}
				H_{BC}^i(M)\oplus R&\text{if }i=j+1\\
				H_{BC}^i(M)/[c]&\text{if }i=j+2\\
				H_{BC}^i(M)&\text{else.}
			\end{cases}\\
			H_{A}^{i}(M\oplus_c \square[p,q])&\cong\begin{cases}
				H_{A}^{i}(M)/[c]_A&\text{if } i=j+2\\
				H_A^i(M)&\text{else,}
			\end{cases}
		\end{align*}
		where $R$ is a vector space of dimension at most $2$, which vanishes if $[c]_A\neq 0$.	
	\end{enumerate}
	
\end{lem}
\begin{proof}
	First assume $c=\del\delbar b$ and denote by $a$ a generator for $\square$, by $x$ a generator for $\cL[p+1,q+1]$. Then the isomorphism $M\oplus_c\square[p,q]\cong M\oplus L[p+1,q+1]$ is induced by the identity on $M$ and $b\mapsto a-x$. The cohomology calculation is then clear.
	
	If $[c]\neq 0$, consider the short exact sequence \[0\to M\to M\oplus_c \square[p,q]\to L[p+1,q+1]\to 0,\]which has an associated long exact sequence. We write down the relevant portion, where $H=H_{BC}$ or $H=H_A$:
	\[
	\dots\to H^i(\bullet[p+1,q+1])\to H^i(M)\to H^i(M\oplus_c\square[p,q])\to H^i(\Lpic[p+1,q+1])\to \dots,
	\]
	where we have used $L^{-1}(\Lpic[p+1,q+1])=\bullet[p+1,q+1]$. The left map sends the generator to $[c]$. Further, we have $H^{j+1}_{BC}(\Lpic[p+1,q+1])=\C^2$, $H^{j}_A(\Lpic[p+1,q+1])=\C$ and $H^i(\Lpic[p+1,q+1])=0$ else. It thus remains to analyze the right map in the two cases where $H^i(\Lpic[p+1,q+1])\neq 0$. If $[c]_{BC}\neq 0$, i.e. $c\not\in \im(\del\delbar)$, we see that it is trivial when mapping to $H^{j}_A(\Lpic[p+1,q+1])$ and under the additional assumption that $[c]_A\neq 0$, i.e. $c\not\in(\im\del+\im\delbar)$, we see that it is trivial when mapping to $H_{BC}^{j+1}$.
\end{proof}

\begin{ex}The following sketch illustrates a situation where $\dim R=2$. Dots and diamonds stand for one-dimensional vector spaces and lines for isomorphisms.
\[
\img[1.5]{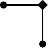}~\oplus_{\img{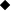}}~\img[1.5]{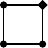}~\cong~\img[1.5]{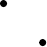}~\oplus~  \img[1.5]{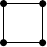}
\]
It is an instructive exercise to draw similar pictures for situations with $\dim R=1$ or $\dim R=0$.
\end{ex}

\subsection{Pluripotential quasi-isomorphisms}\label{sec: bigraded quisos}
We give a more in-depth study of maps in $\bico$ that induce isomorphisms in ${\Ho(\bico)}$ and prove Theorem \ref{thmintro: Quisos} from the introduction.

\begin{lem}\label{lem: characterization acyclicity}
Let $A$ be a bicomplex. The following statements are equivalent:
\begin{enumerate}
\item $A\cong 0$ in ${\Ho(\bico)}$.
\item $A$ is a direct sum of squares.
\item $H_{BC}(A)=0$.
\item $H_A(A)=0$.
\end{enumerate}
If $A$ is locally bounded, these statements are further equivalent to:
\begin{enumerate}[resume]
	\item $H_{\del}(A)=H_{\delbar}(A)=0$.
\end{enumerate}
\end{lem}
\begin{proof}
The equivalence of the first two statements follows from the definitions, keeping in mind that projective objects are exactly direct sums of squares. To show $(2)\Leftrightarrow (3)$, (resp. $(4)$), use \Cref{thm:decomposition} and that $H_{BC}$ (resp. $H_A$) commutes with direct sums and check on every bicomplex separately. $H_{BC}$ (resp. $H_A$) is nonzero exactly on dots and corners of zigzags with incoming (resp. outgoing) arrows. For the statement on row and column cohomology, note that these are nonzero exactly on `endpoints' of zigzags.
\end{proof}
\begin{rem}
	One can avoid using the full \Cref{thm:decomposition} and prove as in \cite[Prop. 5.17]{DGMS75} by elementary means that under the assumption of, say, $H_{BC}(A)=0$, one has a decomposition $A=\bigoplus S_{p,q}$, where $S_{p,q}$ denotes the subcomplex generated by a complement of $\ker\del\delbar\subseteq A^{p,q}$.
\end{rem}
\begin{rem}
An unbounded projective $A$ still satisfies $H_{\del}(A)=H_{\delbar}(A)=0$. However, the converse is not true, as one can see by considering a two-sided infinite zigzag. 
\end{rem}

\begin{thm}\label{prop: characterization quisos}
Let $f:A\to B$ be a map of bicomplexes. The following statements are equivalent:
\begin{enumerate}
\item\label{cond: Iso in SC} $f$ is a bigraded chain homotopy equivalence, i.e. an isomorphism in ${\Ho(\bico)}$.
\item\label{cond: cone squares} $\cone(f)$ is a direct sum of squares.
\item\label{cond: zigs iso} For any choice of decomposition $A=A^{sq}\oplus A^{zig}$ and $B=B^{sq}\oplus B^{zig}$ as in \Cref{cor: minimal + contractible}, the map $f:A^{zig}\to B^{zig}$ is an isomorphism.
\item\label{cond: A+BC iso} $H_{BC}(f)$ and $H_A(f)$ are isomorphisms.
\end{enumerate}
If $A,B$ are locally bounded, these are further equivalent to $H_{\del}(f)$ and $H_{\delbar}(f)$ being isomorphisms.
\end{thm}

\begin{definition}
	A map satisfying the above equivalent conditions will be called a \textbf{pluripotential quasi-isomorphism} or a \textbf{bigraded quasi-isomorphism}.
\end{definition}
We introduce two names for the same concept for the following reasons: The term bigraded quasi-isomorphism is ambiguous since it could also refer to a map of bicomplexes which is a usual quasi-isomorphism (i.e. it induces an isomorphism in total cohomology). On the other hand, it has also been used with the current meaning e.g. in \cite{SW22}, \cite{MS22}. The term pluripotential quasi-isomorphism does not have this ambiguity and conveys the crucial point of \Cref{prop: characterization quisos} \eqref{cond: A+BC iso} that information on existence and uniqueness to the $\partial\bar\partial$-equation is preserved along these quasi-isomorphisms. In this paper, we will use both terms interchangeably, favouring `bigraded' in the abstract algebraic setup and `pluripotential' in context of complex manifolds.

\begin{proof}
The equivalence of \Cref{cond: Iso in SC} and \Cref{cond: cone squares} is an instance of the following general fact in triangulated categories that follows directly from the axioms: A map is an isomorphism if and only if its cone isomorphic to zero. For the equivalence with \Cref{cond: zigs iso}, note that $A^{zig}\to A$ and $B\to B^{zig}$ are isomorphisms in ${\Ho(\bico)}$.

Next, consider the long exact sequence given by \Cref{cor: les Schweitzer}
\begin{multline}
...\to H_{A}^{p-1,q-1}(A)\to H_{A}^{p-1,q-1}(B)\to H_{A}^{p-1,q-1}(\cone(f))\\\to H_{BC}^{p,q}(A)\to H_{BC}^{p,q}(B)\to ...
\end{multline}
and using \ref{lem: characterization acyclicity}, we see that \Cref{cond: Iso in SC} is equivalent to \Cref{cond: A+BC iso}.

The statement about row and column cohomology follows again using the associated long exact sequences as in \Cref{rem: Dolbeault les} and \Cref{lem: characterization acyclicity}.
\end{proof}

\begin{definition}[\cite{Ste21}]
Let $V$ be an additive category. An additive functor $H:\bico\to V$ is called cohomological if it satisfies $H(P)=0$ for any direct sum of squares $P$.
\end{definition}
Typical examples of such functors for us are $H_{BC}$, $H_A$, $H_{\delbar}$, the Schweitzer cohomologies, etc. Additive functors commute with finite direct sums. If $H$ commutes even with infinite direct sums, it is enough to require $H(\square[p,q])=0$ for $p,q\in\Z$. In practice, all our examples will be of this form.
\begin{rem}
	In \cite{Ste21}, an additional finiteness condition and linearity instead of additiveness were imposed, both of which are irrelevant here. In \cite{Ste18} the same notion appeared without name and with the inaccurate wording `$H$ vanishes on squares' which should either be read as `$H$ vanishes on direct sums of squares' or be complemented by the assumption that $H$ commutes with arbitrary direct sums.
\end{rem}

We recover the following result, proved in \cite{Ste18} for bounded complexes by different means. 


\begin{cor}\label{cor: coh functors and quisos}
Let $f:A\to B$ be a map of bicomplexes s.t. $H_{BC}(f)$ and $H_A(f)$ are isomorphisms. Then $H(f)$ is an isomorphism for any cohomological functor. 

If $A,B$ are locally bounded, it suffices to require $H_{\del}(f)$ and $H_{\delbar}(f)$ to be isomorphisms.
\end{cor}

\begin{proof}
It is enough to show that $H$ factors through ${\Ho(\bico)}$. By the universal property of localization this is the case if it sends bigraded homotopy equivalences to isomorphisms in $V$. Since $H$ sends projective objects to zero it sends any map factoring over projectives to zero. Then, for then for any bigraded homotopy equivalence $f$ with quasi-inverse $g$, one has $0=H(f\circ g-\Id)=H(f)H(g)-\Id$ and analogously $H(g)H(f)=\Id$.\end{proof}

\begin{rem}[Terminology]
	In \cite{Ste18} and following works, maps between bounded complexes such that $H_{\del}(f)$ and $H_{\delbar}(f)$ are isomorphisms were called $E_1$-isomorphisms.\\
	In the context of triangulated categories, there is a different notion of cohomological functor, namely an additive functor to an abelian category which sends triangles to long exact sequences. This is a stronger condition than what is considered here, where one only has that functor $H$ factors through the (triangulated) homotopy category, but no assumption about the behaviour on long exact sequences is made.
\end{rem}

%
%

\subsection{The real motives of complex manifolds}

Given any complex manifold $X$, the bicomplex $A_X$ defines an object in $\R\bico$. One may think of the associated object in $\Ho(\R\bico)$ as the `real motive' of $X$. In fact, by definition, any cohomological functor factors through $\Ho(\R\bico)$ and, as we explain in this section, analytically equivalent maps (resp. correspondences) define the same map in $\Ho(\R\bico)$. This answers a question the author was asked by C. Deninger and S. Boucksom (independently).

\begin{lem}\label{lem: Dolbeault Kuenneth and duality}
	Let $X,Y$ be compact connected complex manifolds. The following maps are isomorphisms in $\Ho(\R\bico)$:
	\begin{enumerate}
		\item (K\"unneth formula) $A_X\otimes A_Y\to A_{X\times Y}$, given by $\omega\otimes \eta\mapsto \pr_X^*\omega\wedge\pr_Y^*\eta$.
		\item (Duality) $A_X\to DA_X$, given by $\omega\mapsto\int_X \omega\wedge\_$.
	\end{enumerate}
\end{lem}
\begin{proof}
	Both maps are compatible with the real structures. As noted in \cite{Ste18}, they are $E_1$-isomorphisms as a consequence of the Dolbeault K\"unneth formula, resp. Serre duality \cite{Se55}. Since all involved complexes are bounded, this means they are also pluripotential quasi-isomorphisms by \Cref{prop: characterization quisos}.
\end{proof}



Let us now assume $X$ to be compact and connected and let $Cyc^p(X)$ denote the free abelian group on closed irreducible, reduced complex analytic subspaces of $X$ of pure codimension $p$. We have a map $Cyc^p(X)\to (D_nA_X)^{p,p} $ which sends $Z\subseteq X$ to its current of integration $[Z]:=(\eta\mapsto \int_Z\eta|_Z)\in (D_nA_X)^{p,p}$. This current is obviously $\del$ and $\delbar$-closed and we can compose with a projection to obtain a `fundamental class' map
\begin{align}\label{cycle class map}
	\begin{split}
Cyc^p(X)&\longrightarrow H_{BC}^{p,p}(D_nA_X)\cong H_{BC}^{p,p}(X)\\
Z&\longmapsto [Z].
\end{split}
\end{align}
Assume we are given a compact connected complex manifold $P$, an analytic subspace $W\subseteq X\times P$ of pure codimension $p$ and two points $a,b\in P$ s.t. the fibres $X_a:=\pi_P^{-1}(a)$, $X_b:=\pi_P^{-1}(b)$ of the projection $\pi_P:X\times P\to P$ intersect $W$ transversally. Then $Z_a:=W\cap X_a$ and $Z_b:=W\cap X_b$ are said to be elementary analytically equivalent. We denote by $Cyc^p_0(X)\subseteq Cyc^p(X)$ the subgroup generated by differences of elementary analytically equivalent cycles. Two cycles $Z,Z'\in Cyc^p(X)$ are said to be analytically equivalent if their difference lies in $Cyc^p_0(X)$.  We call the quotient by this equivalence relation the analytic Chow group: $CH_{an}^p(X):=Cyc^p(X)/\sim_{an}$


\begin{lem}
	If $Z\sim_{an} Z'$ are two analytically equivalent cycles on $X$, their fundamental classes coincide: $[Z]=[Z']\in H_{BC}(X)$. 
\end{lem}
\begin{proof}
	Without loss of generality, we can restrict to the case of elementary analytically equivalent subspaces, so let $P,W,a,b,$ and $Z_a, Z_b$ be as above. Note that $[a]-[b]=0\in H_{BC}^{\dim P,\dim P}(P)$. In fact, by duality it suffices to check this by pairing with $H^{0,0}_A(P)=\C$ and for every constant function $f$ one has $f(a)-f(b)=0$. Then, $0=i^*\pi^*([a]-[b])=[Z_a]-[Z_b]$, where $i:W\to X\times P$ is the inclusion and we have used that $i^*([X_a])=[Z_a]$ and $i^*([X_b])=[Z_b]$ (c.f. \cite{Wu22}, which contains an argument even for integral Bott-Chern cohomology).
\end{proof}
\begin{cor}
	For any compact complex manifold $X$, \eqref{cycle class map} induces a well-defined map $CH_{an}^p(X)\to H^{p,p}_{BC}(X)$.
\end{cor}

\begin{cor}
	For two compact complex manifolds $X,Y$ of dimensions $n,m$, there is a map
	\begin{align*}
		CH_{an}^p(X\times Y)\longrightarrow 
		&[A_Y[p-m,p-m],A_X].
\end{align*}\end{cor}
\begin{proof}
	This follows from the following chain of canonical identifications:
	\begin{align*}
		H_{BC}^{p,p}(X\times Y)&\cong [\bullet[p,p],A_{X\times Y}]\\
		&\cong [\bullet[p,p],A_X\otimes A_Y]\\
		&\cong [DA_Y[p,p],A_X]\\
		&\cong [A_Y[p-m,p-m], A_X].
	\end{align*}
Here, we have used \Cref{lem: Dolbeault Kuenneth and duality} in the second and fourth line. The third line uses \eqref{eqn: tensor right adjoint stable cat}, which is applicable since $A_Y$ is quasi-isomorphic to a finite dimensional bicomplex by Theorem \ref{thm:decomposition} and the fact that Dolbeault cohomology is finite dimensional and nonzero on any bounded zigzag (c.f. \cite{Ste18}).
\end{proof}

E.g. if $f:X\to Y$ is a map and $\Gamma_f\subseteq X\times Y$ its graph (so it has dimension $n$ and codimension $p=m$), one obtains a class in $[A_Y,A_X]$ which coincides with the pullback map. In particular, analytically equivalent maps give rise to homotopic maps in $\bico$ and hence to the same pullback maps in every cohomological functor. The above Corollary also shows that the same holds for maps induced by correspondences.

\subsection{ABC-cohomology and minimal models for bicomplexes} For any bicomplex $A$, by \Cref{cor: minimal + contractible} there is a decomposition $A=A^{zig}\oplus A^{min}$ into a minimal and a contractible part. Such a decomposition is a choice and generally not compatible with maps. In this section we give a functorial meaning to the summand $A^{zig}$.

We note that the differentials $\del$ and $\delbar$ induce maps of bidegree $(1,0)$ and $(0,1)$:
\[
\del,\delbar: H_A(A)\longrightarrow H_{BC}(A)
\]
and the zero map on $H_{BC}$. Thus, we can consider $H_{A}\oplus H_{BC}$ as a functor 
\[
\text{bicomplexes}\longrightarrow\text{minimal bicomplexes}.
\] However, in general $H_{A}(A)\oplus H_{BC}(A)$ is not isomorphic to $A$ in ${\Ho(\bico)}$. For instance, if $A=\C$ is a dot, concentrated in a single bidegree, then 
\[H_A(A)\oplus H_{BC}(A)=\C^2\neq\C= A\]
and the two sides can also not be isomorphic in ${\Ho(\bico)}$ since they have different cohomology. We will now present various ways to fix this `double counting'. For this, it is convenient to introduce reduced versions of $H_A$ and $H_{BC}$. Consider the natural transformation induced by the identity
\begin{equation}\label{eqn: BC to A}
H_{BC}\longrightarrow H_A.
\end{equation}
Now we define $\widetilde{H}_{BC}$, resp. $\widetilde{H}_{A}$, resp $H_{dot}$ the functors which send a bicomplex to the (objectwise) kernel, resp. the cokernel, resp. the image of \eqref{eqn: BC to A}, which we may describe explicitly as:
\begin{align}
\widetilde{H}_{BC}&=\frac{\im\del\cap\ker\delbar+\im\delbar\cap\ker\del}{\im\del\delbar}\\ 
\widetilde{H}_{A}&=\frac{\ker\del\delbar}{\im\del+\im\delbar + \ker\del\cap\ker\delbar}\\
H_{dot}&=\frac{\ker\del\cap\ker\delbar}{\im\del\cap\ker\delbar+\im\delbar\cap\ker\del}
\end{align}
By definition, for any $A\in \bico$, there is an exact sequence
\[
0\longrightarrow \HBCred (A)\longrightarrow H_{BC}(A)\longrightarrow H_A(A)\longrightarrow \HAred(A)\longrightarrow 0,
\]
and one checks that there are natural identifications
\begin{align}
\HBCred&=\im(H_A\oplus H_A\overset{\del+\delbar}{\longrightarrow} H_{BC})\\\label{eqn: HBCred as image}
 H_{dot}&=\ker( H_A\overset{(\del,\delbar)}{\longrightarrow} H_{BC}\oplus H_{BC}).
\end{align}

\begin{definition}
For any bicomplex $A$, the $ABC$-cohomology is the bicomplex
\[
H_{ABC}(A):=H_{A}(A)\oplus \HBCred(A)
\]
and the $BCA$-cohomology is the bicomplex
\[
H_{BCA}(A):=H_{BC}(A)\oplus \HAred(A).
\]
\end{definition}

For any bicomplex $A$ with $\del\delbar\equiv 0$, we have a short exact sequences
\begin{align}
	0\longrightarrow H_{BC}(A)\longrightarrow A\longrightarrow \HAred(A)\longrightarrow 0,\label{eqn: BCA}\\
	0\longrightarrow \HBCred(A)\longrightarrow A\longrightarrow H_A(A)\longrightarrow 0.\label{eqn: ABC}
\end{align}

We note that $\del\delbar\equiv 0$ on $H_{ABC}$ and $H_{BCA}$, so both constitute functors
\[
\text{bicomplexes}\longrightarrow\text{minimal bicomplexes}.
\] 
In fact, the result is quasi-isomorphic to $A$
\begin{prop}
	For any bicomplex $A$, there are (non-canonical) isomorphisms in $\Ho(\bico)$
	\[
	A\cong H_{ABC}(A)\cong H_{BCA}(A).
	\]
\end{prop}
\begin{proof}
	One checks by an elementary computation that all expressions are zero on squares and the identity on zigzags. The result then follows from \Cref{thm:decomposition}.
\end{proof}
\begin{rem}[Duality]
	For any bicomplex one has a canonical identification $H_{BCA}(DA)\cong DH_{ABC}(A)$ and thus, by \Cref{lem: Dolbeault Kuenneth and duality}, for any compact complex $n$-fold $X$ there are canonical isomorphisms $H_{BCA}^{p,q}(X)\cong (H_{ABC}^{n-p,n-q}(X))^\vee$. To obtain a self-dual version, one may consider $\HBCred\oplus H_{dot}\oplus \HAred$.
\end{rem}
\subsection{K\"unneth formulae}
For two bicomplexes $A,B$, there are natural maps

\[	H_{BC}(A)\otimes H_{BC}(B)\longrightarrow H_{BC}(A\otimes B)\]
and
 \[	H_{BC}(A)\otimes H_A(B)\oplus H_{A}(A)\otimes H_{BC}(B)\longrightarrow H_A(A\otimes B).\]

Contrary to what one might guess, these are in general neither injective nor surjective.

\begin{ex}
	If $A=L=\langle 1,\del ,\delbar \rangle$, $B=L^{-1}=\langle \del,\delbar,\del\delbar\rangle$, then Bott-Chern cohomology of the tensor product $A\otimes B$ is both one-dimensional and generated by the class $[\delbar\otimes\del + 1\otimes\del\delbar + \del\otimes\delbar]$. On the other hand, $H_{BC}(A)=\langle[\del],[\delbar]\rangle$ and $H_{BC}(B)=\langle[\del\delbar]\rangle$ and the first map is zero for bidegree reasons. 
\end{ex}
The K\"unneth formulae for Bott-Chern and Aeppli cohomology will give explicit expressions for the kernel and cokernel of these maps. Before going through the somewhat technical statement and proof, here is the simple idea: If $A=A^{zig}\oplus A^{sq}$ and $B=B^{zig}\oplus B^{sq}$ as in \Cref{cor: minimal + contractible}, then since the tensor product with a contractible complex is contractible the zigzag part of $A\otimes B$ can be computed from $A^{zig}\otimes B^{zig}$. Since the zigzag part can be unambiguously defined via $ABC$-cohomology, we expect an formula of the form
\[
H_{ABC}(X\times Y)\cong H_{ABC}(H_{ABC}(X)\otimes H_{ABC}(Y))
\]
Since $H_{ABC}$ is a sum of two components, this would split into two formulas, one for $H_A$ and one for $\HBCred$. The same comments apply to $H_{BCA}=H_{BC}\oplus\HAred$. We will now unwrap the parts referring to $H_A$ and $H_{BC}$ and define explicit maps:
\begin{thm}[ABC K\"unneth formulae]\label{thm: algebraic ABC Kuenneth}
	Let $A,B$ be bounded bicomplexes. 
	\begin{enumerate}
		\item There is a natural short exact sequence
		\begin{equation}\label{eqn: SES BC Kuenneth}
			0\longrightarrow \frac{H_{BC}(A)\otimes H_{BC}(B)}{\del\delbar(H_A(A)\otimes H_A(B))}\longrightarrow H_{BC}(A\otimes B)
			\longrightarrow K\longrightarrow 0 
		\end{equation}
		where $K$ is the following space
		\[
		K:=(\ker\del\cap\ker\delbar)(\HAred(A)\otimes H_{BC}(B)\oplus H_{BC}(A)\otimes \HAred(B))
		\]
			\item There is a natural short exact sequence
		\begin{equation}\label{eqn: SES Aeppli Kuenneth}
			0\longrightarrow L\longrightarrow H_A(A\otimes B)\longrightarrow (\ker\del\delbar)(H_A(A)\otimes H_A(B))\longrightarrow 0,
		\end{equation}
		where 
		\[
		L:=\frac{\HBCred(A)\otimes H_A(B)\oplus H_A(A)\otimes \HBCred(B)}{\del(H_A(A)\otimes H_A(B))+ \delbar(H_A(A)\otimes H_A(B))}
		\]
	\end{enumerate}
\end{thm}

\begin{proof}
	Concerning \eqref{eqn: SES BC Kuenneth}, we define the left-hand map via 
	\[
	[a]_{BC}\otimes [b]_{BC}\mapsto [a\otimes b]_{BC}
	\] 
	and denote it by $i$. It is well-defined since, up to sign, \[
	\del\delbar([c]_{A}\otimes [d]_A)=[\del c]_{BC}\otimes[\delbar d]_{BC}\pm[\delbar c]_{BC}\otimes [\del d]_{BC}\mapsto [\del\delbar(c\otimes d)]_{BC}=0.
	\]
	On the other hand, there is a natural homomorphism from $p:K\to \coker(i)$ as follows:
	
	Consider a class $\mathfrak{c}\in K$. It can be written as
	\[
	\mathfrak c=\sum_{i} [a_i]_{\tilde A}\otimes [b_i]_{BC}+[c_i]_{BC}\otimes[d_i]_{\tilde A},
	\] 
	where we assume the individual factors to be of pure degree.
	The conditions $\del\mathfrak{c}=0$ translate into the existence of $e^\del_j,f^\del_j,g^\del_j,h^\del_j$ of pure degree such that:
	\[
	\sum_i\del a_i\otimes b_i+(-1)^{|c_i|}c_i\otimes \del d_i=\sum \del\delbar e_j^\del\otimes f_j^\del+~g^\del_j\otimes \del\delbar h^\del_j
	\]
	and similarly $e^\delbar_k,f^\delbar_k,g^\delbar_k,h^\delbar_k$ satisfying the analogous equation with the roles of $\del$ and $\delbar$ reversed. We may then put
	\begin{align*}
	p(\mathfrak c):=&[\sum_i a_i\otimes b_i+c_i\otimes d_i\\
	&-\sum_j(\delbar e^\del_j\otimes f^\del_j+(-1)^{|g^\del_j|} g^{\del}_j\otimes \delbar f^\del_j)\\
	&+\sum_k(\del e^\delbar_k\otimes f^\delbar_k+(-1)^{|g^\delbar_k|} g^{\delbar}_k\otimes \del f^\delbar_k)]
	\end{align*}
	One verifies that this defines indeed a class in $\coker(i)$ which is independent of the choices made and that the resulting map is a homomorphism. For example, if one makes a different choice $\tilde{e}_j^\del=e_j^\del + e'$ with a $\del\delbar$-closed element $e'$, the representative for $p(\mathfrak c)$ changes by $\delbar e'\otimes f_j^\del\in H_{BC}(A)\otimes H_{BC}(B)$. 
	
	Now that the maps are in place, it remains to show injectivity of $i$ and that $p$ is an isomorphism. For this, one may assume $A$ and $B$ indecomposable in which case this is a straightforward calculation. We will only indicate the results for every possible pair of indecomposable bicomplexes:
	
	If either $A$ or $B$ is a square, all terms in the short exact sequence vanish, so there is nothing to show. If $A$ is a dot and $B$ an arbitrary zigzag, then $A\otimes B$ is a shifted version of $B$, and $\del(H_A(A))=\delbar(H_A(A))=0$ so the denominator of the left hand term and $K$ vanish and $i$ is an isomorphism. Finally, assume $A,B$ are both zigzag of length $\geq 2$, concentrated in degrees $k(A),k(A)+1$ and $k(B),k(B)+1$. Then the tensor product $A\otimes B$ is a bicomplex concentrated in degrees $l,l+1,l+2$ with $l=k(A)+k(B)$. In degree $l$, all terms in \eqref{eqn: SES BC Kuenneth} vanish. In degree $l+1$, the left hand side vanishes and $p:K\to H_{BC}(A\otimes B)$ is an isomorphism (note also that there is no choice involved in the definition of $p$ in this total degree). Finally, in degree $l+2$, $K=0$ and $i$ is an isomorphism.


	The proof for \eqref{eqn: SES Aeppli Kuenneth} is similar, but here the definition and well-definedness of the maps is more immediately clear. Again one may verify exactness on a case-by-case inspection of the combinations of indecomposable bicomplexes, which yields the following results: For $A$ or $B$ a square, the whole sequence is zero. For $A$ a dot and $B$ any zigzag, $\del H_A(A)=\delbar H_A(A)=\HBCred(A)=0$ so the left hand side vanishes and the right hand map is an isomorphism $H_A(A\otimes B)\cong H_A(A)\otimes H_A(B)$. Now, consider the case of both $A,B$ zigzags of length $\geq 2$ in degrees $k(A),k(A)+1$, resp. $k(B),k(B)+1$, so that $A\otimes B$ is concentrated in degrees $l,l+1,l+2$ with $l=k(A)+k(B)$. All spaces in the sequence will be zero except possibly in degrees $l,l+1$. In degree $l$, the left hand side is zero and the right hand map is an isomorphism. In degree $l+1$, the right hand side is zero and the left hand map is an isomorphism.
\end{proof}

\begin{cor}\label{cor: geometric ABC Kuenneth}Let $X,Y$ be compact complex manifolds and consider the projections $\pr_X,\pr_Y$ to both factors. Then there are short exact sequences
	\begin{equation}\label{eqn: SES-BC mfds}
		0 \longrightarrow\frac{H_{BC}(X)\otimes H_{BC}(Y)}{\del\delbar(H_A(X)\otimes H_A(Y))}\overset{\pr^*_X\times \pr_Y^*}{\longrightarrow} H_{BC}(X\times Y)
		\longrightarrow K\longrightarrow 0,
	\end{equation}
	and 
	\begin{equation}\label{eqn: SES Aeppli Kuenneth mfds}
	0\longrightarrow L\longrightarrow H_A(X\times Y)\longrightarrow (\ker\del\delbar)(H_A(X)\otimes H_A(Y))\longrightarrow 0,
\end{equation}

	where $K,L$ are defined as in \Cref{thm: algebraic ABC Kuenneth}.
\end{cor}
\begin{proof}
	This follows from \Cref{thm: algebraic ABC Kuenneth}, and \Cref{lem: Dolbeault Kuenneth and duality}.
\end{proof}
Let us specialize the Bott-Chern formula to the case of degrees $(p,0)$, i.e. to closed holomorphic forms. We then recover the following result:
\begin{cor}
	For compact complex manifolds $X,Y$, there are injections
	\[
	\bigoplus_{a+b=p} H^{a,0}_{BC}(X)\otimes H^{b,0}_{BC}(Y)\hookrightarrow H_{BC}^{p,0}(X\times Y)
	\]
	for any given $p\in\Z$.
\end{cor}
Conjugation gives the analogous result in degrees $(0,p)$. In degree $(1,0)$, we recover the following result, contained (without surjectivity in $H_A$) in \cite{CR22}.
\begin{cor}
	For compact complex manifolds $X,Y$, there are isomorphisms
	\[
	H_{BC}^{0,1}(X)\oplus H_{BC}^{0,1}(Y)\cong H^{0,1}_{BC}(X\times Y)
	\]
	and
	\[
	H_A^{0,1}(X)\oplus H_A^{0,1}(Y)\cong H^{0,1}_{A}(X\times Y)
	\]
\end{cor}

\begin{proof}
	Note that we have $H_{BC}^{0,0}(X)=\C$. Thus, if we restrict to degree $(0,1)$, the left hand side of \eqref{eqn: SES-BC mfds} becomes $H_{BC}^{0,1}(X)\oplus H_{BC}^{0,1}(Y)$. On the other hand, we have $\HAred^{0,0}(X)=0$ and so 
	\begin{align*}	K&=(\ker\del\cap\ker\delbar)(\HAred^{0,1}(Y)\oplus \HAred^{0,1}(Y))\\
		&=(\ker\del\cap\ker\delbar)(\HAred^{0,1}(Y))\oplus (\ker\del\cap\ker\delbar)(\HAred^{0,1}(Y))=0.
	\end{align*}
	This shows the first isomorphism. For the second, we note that again because $\HAred^{0,0}(X)=0$ we have $\HBCred^{0,1}(X)=\HBCred^{0,0}(X)=0$ and so the left hand side of \eqref{eqn: SES Aeppli Kuenneth mfds} in degree $(0,1)$ is zero. Since $H_{BC}^{0,0}(X)\cong H_A^{0,0}(X)=\C$ the right hand side is equal to $H_A^{0,1}\oplus H_A^{0,1}$.	
\end{proof}

\begin{rem}
	Another way of proving the Bott-Chern case of this Corollary is to note that there is an isomorphism of holomorphic vector bundles $p_X^*\Omega_X^1\oplus p_Y^*\Omega_Y^1\cong \Omega_{X\times Y}^1$ and hence an identification of global sections (holomorphic forms), which restricts to an identification of the spaces of closed holomorphic forms. This gives the result in degree $(1,0)$, which, by conjugation, is equivalent to the statement in degree $(0,1)$.
\end{rem}

%
The same kind of arguments show:
\begin{cor}
	If $X,Y$ are compact complex manifolds and $X$ satisfies the $\del\delbar$-Lemma, the pullbacks $p_X^*, p_Y^*$ induce isomorphisms
	\begin{align*}
		H_{BC}(X)\otimes H_{BC}(Y)&\cong H_{BC}(X\times Y)\\
		H_{A}(X)\otimes H_A(Y)&\cong H_A(X\times Y).
	\end{align*}
\end{cor}

\section{Wedge products}\label{sec: wedge}

\subsection{Commutative bigraded bidifferential algebras}
By a (graded) commutative bigraded bidifferential algebra (short: cbba) over $\C$, we mean a bicomplex $(A^{\Cdot,\Cdot},\del,\delbar)$ with the additional data of a product $\wedge:A\times A\to A$ of bidegree $(0,0)$ and a unit $\C\to A$ that makes $A$ into a unital graded-commutative algebra (with respect to the total degree) and for which $\del$ and $\delbar$ satisfy the Leibniz-rule. We denote the category of all cbba's by $\cbba$. 

\subsubsection{Free-forgetful adjunctions}For any bicomplex $B$, we can form the free cbba $\Lambda B$ on $B$, i.e. the free bigraded graded-commutative algebra on $B$ as a vector space, with differentials induced by those of $B$. Conversely, we have a forgetful functor sending a cbba to its underlying bicomplex. These constructions give a pair of adjoint functors
\begin{equation}\label{eqn: adjunction bico cbba}
	\Lambda:\bico\leftrightarrows\cbba:U
\end{equation}
We will also consider the category $\cbba_0$ of augmented cbba's, i.e. of pairs $(A,\varepsilon)$ where $A\in \cbba$ and $\varepsilon:A\to \C$ is a map in $\cbba$. For instance, for any complex manifold $X$ and a point $x\in X$, the evaluation map $\varepsilon_x: A_X\to \C$ is an augmentation. For any $(A,\varepsilon)$ in $\cbba_0$, we denote by $A^+=\ker\varepsilon$ the augmentation ideal. Note that a free cbba $\Lambda B$ has a canonical augmentation induced by $B\ni b\mapsto 0$. This yields an adjunction
\begin{equation}\label{eqn: adjunction bico cbba augmented}
	\Lambda: \bico\leftrightarrows\cbba_0:(~)^+.
\end{equation}

We will further consider the full subcategories $\cbba^{\geq 0}$ and $\cbba^{f.q.}$ of algebras concentrated in nonnegative total degree, resp. the first quadrant, i.e. such that $A^{p,q}=0$ whenever $p+q<0$ (resp. $A^{p,q}=0$ for $p<0$ or $q<0$). For any cbba $A$, the complex conjugate $\bar{A}$ of the underlying bicomplex is again naturally a cbba and a real structure on $A$ is a $\C$-linear isomorphism $\sigma: A\cong \bar{A}$. We denote the category of pairs $(A,\sigma)$ by $\R\cbba$ and the augmented version by $\R\cbba_0$. There are straightforward versions of the free-forgetful adjunction above for the categories $\R\bico\leftrightarrows \R\cbba$ and the bounded-below counterparts of the complex and real categories.

\subsubsection{Notions of connectedness}
A cbba $A$ is called connected if it is concentrated in non-negative total degrees and the unit $\C\to A^0$ is an isomorphism. It is called simply connected if it is connected and $A^1=0$. For example, the free algebra on a positively graded bicomplex is connected. $A$ is called cohomologically connected if $H_A^{< 0}(A)=0$ and $\C\cong H^0(A)$. A compact complex manifold $X$ is connected iff $A_X$ is cohomologically connected. $A$ is called cohomologically simply connected if it is cohomologically connected and $H^1_A(A)=0$. A compact complex manifold $X$ is called holomorphically simply connected if $A_X$ is cohomologically simply connected. Using \Cref{thm:decomposition}, one checks that for $X$ holomorphically simply connected one has $H^1_{\delbar}(X)=0$ and, in particular, $H_{dR}^1(X)=0$.

\subsection{A model category structure on $\cbba$}
To define a model category structure on $\cbba$, we transfer the model category structure from $\bico$ via the adjunction in \eqref{eqn: adjunction bico cbba}. More precisely, we claim to obtain a model category structure on $\cbba$ if we define a map of cbba's to be

\begin{enumerate}
	\item a fibration if it is surjective,
	\item a weak equivalence if it is a bigraded quasi-isomorphism,
	\item a cofibration if it satisfies the left-lifting property with respect to all acyclic fibrations.
\end{enumerate}
Recall that a fibration is said to be acyclic if it is also a weak equivalence and that a map $f:A\to B$ has the left-lifting property with respect to a map $p:C\to D$ if for any commutative diagram of solid arrows
\[
\begin{tikzcd}
	A\arrow[r]\arrow[d,"f"]&C\arrow[d,"p"]\\
	B\ar[ru,"q",dashed]\arrow[r]&D,
\end{tikzcd}
\]
there exists a map $q:B\to C$ making the two triangles commute. 
\begin{thm}\label{thm: model cat cbba}
Defining the fibrations, weak equivalences and cofibrations as above yields a model category structure on $\cbba$ (resp. $\R\cbba$).
\end{thm}
Since for any model category, the category of objects with a map to some fixed object has an induced model category structure, we immediately obtain:

\begin{cor}
	There are induced model category structures on the categories of augmented cbba's $\cbba_0$ and $\R\cbba_0$.
\end{cor}

To explain the proof of \Cref{thm: model cat cbba}, we first look at some particular examples of cofibrations, which are the bigraded versions of Hirsch extensions (c.f. \cite{GM13})

\begin{definition} Let $A$ be a cbba.  A (bigraded) \textbf{Hirsch extension} of $A$ is an inclusion
	\begin{align*} 
		A&\rightarrow A\otimes\Lambda V\\
		a&\mapsto a\otimes 1,
	\end{align*}
where $V=V^k=\bigoplus_{p+q=k}V^{p,q}$ is a bigraded vector space concentrated in a single total degree $k$ and the bigraded algebra on the right is equipped with a differential that makes the inclusion into a map of cbba's and restricts to a linear map $d:V\rightarrow\ker d|_A$ of degree $(1,0)+(0,1)$.
\end{definition}
If $A\in \R\cbba$ was equipped with a real structure, by Hirsch extension we mean the same as above, while also requiring that $A\otimes \Lambda V$ is equipped with an involution extending that on $A$ and fixing $V$.

\begin{lem}[Hirsch extension lifting Lemma]\label{lem: Hirsch Lifting}

		Hirsch extensions are cofibrations.
\end{lem}
\begin{proof}
		Consider a solid commutative diagram of cbba's
\[
\begin{tikzcd}
	A\ar[r, "\tilde{f}_A"]\ar[d]&B\ar[d, "g"]\\
	A\otimes \Lambda V\ar[r,"f_{AV}"]\ar[ru, dashed]&C
\end{tikzcd}
\] where $A\rightarrow A\otimes \Lambda V$ is a Hirsch extension and $g:B\to C$ a surjective bigraded quasi-isomorphism.  We need to show there exists a map $\tilde{f}_{AV}:A\otimes \Lambda V\to B$ extending $\tilde{f}_A$ s.t. $g\circ\tilde{f}_{AV}=f_{AV}$. 
	
	Note that it suffices to describe $\tilde{f}_{AV}$ on $V$. To lighten up notation, let us assume $V$ is one-dimensional and let $v\in V$ be a basis element (in general, one should take a basis $(v_i)$; which is furthermore stable under conjugation if one is working in $\R\cbba$). To define $\tilde{f}_{AV}$ as an algebra map, we could simply pick arbitrary preimages  $\tilde{v}$ under $g$ for the element $f_{AV}(v)$ and set $\tilde{f}_{AV}(v)=\tilde{v}$. However, we have $\del v,\delbar v, \del\delbar v\in A$ and so we have to respect the conditions $\del \tilde{v} = \tilde{f}_A(\del v)$ and $\delbar\tilde{v}=\tilde{f}_A(\delbar v)$.
	
	In a first step, we will show that it is possible to pick preimages $\tilde{v}'$ of $f_{AV}(v)$ s.t. the weaker condition $\del\delbar \tilde{v}'=\tilde{f}_A(\del\delbar v)$ is satisfied: Note that $\del\tilde{f}_A(\del\delbar v)=\delbar\tilde{f}_A(\del\delbar v)=0$ so that we get a well-defined class $[\tilde f_A(\del\delbar v)] \in H_{BC}(B)$. Since $g\tilde f_A(\del\delbar v)=f_{AV}(\del\delbar v)=\del\delbar f_{AV}(v)$, this maps to the zero class in $H_{BC}(C)$. But $H_{BC}(g)$ is injective, so we get some element $b\in B$ with $\del\delbar b= \tilde f_A(\del\delbar v)$. So $g(b)-f_{AV}(v)$ maps to zero under $\del\delbar$. Since $g$ and $H_{A}(g)$ are surjective, this means there is a $\tilde{b}\in B$ with $\del\delbar\tilde b=0$ and $g(\tilde{b})=g(b)-f_{AV}(v)$. Set $\tilde{v}':=b-\tilde b$.
	
	In a next step we modify $\tilde{v}'$ in a such a way that it satisfies the original conditions. Note that $\del\tilde{v}'-\tilde{f}_{A}(\del v)$ lies in the kernel of $\del$ and $\delbar$ and so defines a class in Bott-Chern cohomology. This class maps to zero under $g$ since $g(\del\tilde v')=\del g(\tilde v')=\del f_{AV}(v)=g \tilde f_A(\del  v)$. Because $H_{BC}(g)$ is injective, this implies that there exists a $b_\del\in B$ s.t. $\del\delbar b_\del=\del\tilde v'-\tilde f_A(\del v)$. By an analogous argument, we get a $b_\delbar\in B$ s.t. $\del\delbar b_\delbar=\delbar\tilde v'-\tilde f_A(\del v)$. Therefore, $\tilde v'':=\tilde v'-\delbar b_\del + \del b_\delbar$ satisfies $\del\tilde v''=\tilde f_A(\del v)$ and $\delbar\tilde v''=\tilde f_A(\delbar v)$. It remains to fix its image under $g$. We do this as before: $g(\tilde v'')-f_{AV}(v)$ lies in the kernel of both $\del$ and $\delbar$ and so defines a Bott-Chern class. Since $g$ and $H_{BC}(g)$ are surjective, we obtain a $\tilde b'\in B$ s.t. $\del\tilde b'=\delbar\tilde b'=0$ and $g(\tilde b')=g(\tilde v'')-f_{AV}(v)$. Defining $\tilde v=\tilde v''-\tilde b'$, we are done.
\end{proof}

\begin{rem}
	There seem to be no obstacles to developing the theory of Hirsch extensions and obstruction cocycles to lifting in analogous fashion to \cite[§10-§11]{GM13}, using the bigraded mapping cone to define relative cohomology. Since for us the above somewhat more elementary arguments suffice, we omit this.
\end{rem}

We write $S(p,q):=\Lambda( \bullet[p,q])$ for the free algebra generated by a dot in degree $(p,q)$ and $T(p,q):=\Lambda (\square[p,q])$ for the free algebra generated by a square with base in degree $(p,q)$. This notation is a bit unfortunate for our purposes and will be used in this section only, to ease comparison with \cite{BG76}.

\begin{cor}\label{cor: cofibrations}
	The following maps are cofibrations, i.e. they satisfy the left lifting property with respect to surjective bigraded quasi-isomorphisms in $\cbba$.
	\begin{enumerate}
		\item\label{it: dot-square cof} The inclusion $S(p+1,q+1)\hookrightarrow T(p,q)$.
		\item\label{it: C-free cof} The inclusion $\C\hookrightarrow\Lambda(C)$ in degree $(0,0)$, where $C$ is any  bicomplex.
		\item\label{it: copr-po-infco cof} coproducts, pushouts and (possibly transfinite) compositions of cofibrations.	\end{enumerate}
\end{cor}
\begin{proof}
\eqref{it: copr-po-infco cof} is a standard result about maps defined by having the lifting property with respect to some class of maps, c.f. \cite[4.5.]{BG76}. \eqref{it: dot-square cof} is a direct consequence of \Cref{lem: Hirsch Lifting}.  For \eqref{it: C-free cof}, first consider $C$ to be indecomposable. Then $C$ is concentrated in finitely many (at most three) total degrees and so $\C\hookrightarrow \Lambda(C)$ can be written as the composition of finitely many Hirsch extensions. The case of a general bicomplex follows from \Cref{thm:decomposition} and \eqref{it: copr-po-infco cof}.
\end{proof}

\begin{proof}[Proof of \Cref{thm: model cat cbba}]
 With the results we have, the proof is very similar to the singly graded case \cite[§4]{BG76}, so we only indicate how to adapt the most interesting axiom, namely that any map $f:A\to B$, can be functorially factored as 
 \[
f=(acyclic~fibration)\circ(cofibration)
\] This will use a slightly non-obvious characterization of weak equivalences (\Cref{lem: A-surj BC-inj} below). The proofs of the other axioms carry over immediately from \cite{BG76}. We will define the intermediate space as a colimit of a diagram
 \begin{equation}\label{diag: colim Bfi}
 	\begin{tikzcd}
 		A\ar[r,"\beta_1"]\ar[d,swap,"f"]&L_f(1)\ar[r,"\beta_2"]\ar[ld,swap,"\psi_1"]&L_f(2)\ar[r,"\beta_3"]\ar[lld,swap,"\psi_2"]&\hdots\\
 		B.
 	\end{tikzcd}
 \end{equation}
 where each $\beta_i$ will be a cofibration (hence also the colimit $A\to L_f:=\operatorname{colim}L_f(n)$). 
 
 Define $L_f(1)$ as follows:
 \[
 L_f(1):= A\otimes\bigotimes_{b\in B^h} T(|b|)
 \]
 Here $B^h=\bigcup_{p,q\in\Z} B^{p,q}\subseteq B$ denotes the set of elements of pure bidegree and for $b\in B^{p,q}$ we write $|b|=(p,q)$ for its bidegree. By construction, $A\to L_f(1)$ is a cofibration as a coproduct of cofibrations and the map $\psi_1:L_f(1)\to B$, given by the projection of the generators of $T(|b|)$ to $b$, is surjective, hence a fibration. 
 
 With the higher $L_f(i)$, we will make $H_A(\psi)$ surjective and kill the lack of injectivity of $H_{BC}(\psi_1)$. We construct them as follows: Given $L_f(i),\psi_i$, consider the set
 \[
 R:=\left\{(w,b)~\bigg|~\substack{\del_1 w=\del_2 w=0,\\ \psi_i(w)=\del\delbar b\\|w|+(1,1)=|b|}\right\}\subseteq (L_f(i))^h\times B^h
 \]
 and define $(L_f(i+1),\beta_{i+1})$ as the pushout
 \[
 \begin{tikzcd}
 	\bigotimes_{(w,b)\in R} S(|w|)\ar[r]\ar[d]& L_f(i)\ar[d,"\beta_{i+1}"]\\
 	\bigotimes_{(w,b)\in R}T(|b|)\ar[r]& L_f(i+1).
 \end{tikzcd}
 \]
 Since the left-hand map is a cofibration as a coproduct of cofibrations, so is $\beta_{i+1}$ and the universal property yields a map $\psi_{i+1}:L_f(i+1)\to B$. By construction, for any element $[w]\in\ker H_{BC}(\psi_i)$, we have $\beta_{i+1}[w]=0$. Thus, $\psi:=\lim\psi_i:L_f\to B$ is injective in $H_{BC}$. Furthermore, already $H_A(\psi_{2})$, and hence $H_A(\psi)$, is surjective (consider the pairs $(0,b)\in R$). That $\psi$ is acyclic now follows from the next Lemma.
\end{proof}
\begin{lem}\label{lem: A-surj BC-inj}
	A map $\varphi\in Map(\bico)$ of bicomplexes s.t. $H_A(\varphi)$ is surjective and $H_{BC}(\varphi)$ is injective is a bigraded quasi-isomorphism.
\end{lem}

\begin{proof}
	If $H_A(\varphi)$ is surjective, also $\HBCred(\varphi)$ is surjective by \eqref{eqn: HBCred as image}. Then the result follows by the four Lemma applied to the diagram 
	\[
	0\to\HBCred(\varphi)\to H_{BC}(\varphi)\to H_A(\varphi)\to H_{BC}(\varphi)\oplus H_{BC}(\varphi).
	\]
\end{proof}

In particular, we obtain a cofibrant replacement functor (c.f. \cite[p.5]{Hov99}), i.e. using the notation from the above proof:

\begin{definition}
	For any cbba $A$ with unit $\eta_A:\C\to A$, the \textbf{cofibrant replacement} is $CA:=L_{\eta_A}$.
\end{definition}

The cofibrant replacement of a cbba is very big and we will be interested in finding smaller cofibrant models. For now we introduce some language:

\begin{definition}\label{def: nilpotent, minimal, model} 
	Let $A\in\cbba_{0}$ be an augmented cbba. 
	\begin{enumerate}
		\item $A$ is called a \textbf{nilpotent} if it is a (possible transfinite) composition of Hirsch extensions. 
		\item $A$ is called \textbf{minimal}, if it is nilpotent and $\im\del\delbar\subseteq A^+A^+$.
		\item A weak equivalence $M\rightarrow A$, where $M$ is a nilpotent algebra is called a \textbf{model} for $A$. If $M$ is minimal, it is called a \textbf{minimal model}.
	\end{enumerate}
\end{definition}
We note that an augmented cbba $A$ is nilpotent iff there is a presentation as a free commutative bigraded algebra $A=\Lambda V$ with a well-ordered basis $\{v_i\}_{i\in I}$ for the space of generators $V$ such that $dv_i\in \Lambda (\operatorname{span}\{v_j\mid j<i\})$. Furthermore, a nilpotent cbba is minimal iff the bicomplex of indecomposables $A^+/A^+A^+$, with induced differentials, is minimal. Again, there are obvious real versions of these definitions. 
By \Cref{lem: Hirsch Lifting} and \Cref{cor: cofibrations}, we obtain:
\begin{cor}[Lifting Lemma]\label{cor: nilpotent implies cofibrant}
	Nilpotent cbba's are cofibrant in $\cbba_0$. 
\end{cor}

\begin{rem}[Bounded variants]\label{rem: bounded vars}
The same kind of constructions used to prove \Cref{thm: model cat cbba} allow one to transfer the model category structures from the non-negatively graded, resp. first quadrant, subcategories of $\bico$ to $\cbba^{\geq 0}$, resp. $\cbba^{f.q.}$, and the corresponding real versions. Namely, one defines again fibrations to be surjective maps, weak equivalences to be bigraded quasi-isomorphisms and cofibrations via the left-lifting property. The reader will have no trouble to verify that the augmentation maps $\Lambda (\bullet[p,q])\to \C$ are cofibrations in $\cbba^{\geq 0}$ for $0\leq p+q\leq1$ and cofibrations in $\cbba^{f.q.}$ whenever $p+q\geq 0$ and $p=0$ or $q=0$. Consequently, one may put $T(p,q):=\C$ in those bidegrees and run the same proof as above for the construction of fibration-cofibration factorizations. 
\end{rem}

\subsection{Homotopy for cbba's and complex manifolds}
We will define and study various `holomorphic' variants of the homotopy groups. 

We start by describing a functorial path object, i.e. a functorial factorization of the maps $A\to A\times A$ into an acyclic cofibration followed by a fibration.
\begin{definition}
	We denote by $\Omega^1_{big}:=\Lambda\langle t,\del t,\delbar t,\del\delbar t\rangle$, where $|t|=(0,0)$. 
\end{definition}

\begin{lem}\label{lem: contractibility of big forms}
	The inclusion $\C\to \Omega^1_{big}$ is a bigraded quasi-isomorphism.
\end{lem}

\begin{proof}
Note that $\Omega^1_{big}$ is nothing but the free algebra on a square $\Lambda(\square)$, which is a direct summand in the tensor algebra over $\square$ which is contractible since nontrivial tensor products of projective objects are projective.
\end{proof}
\begin{rem}
	Alternatively, it follows from $0\simeq \square$ in $\bico$ using that $\Lambda$ as a part of a Quillen adjunction \eqref{eqn: adjunction bico cbba} preserves weak equivalences. Of course, it can also be seen by a direct calculation, see \Cref{ex: algebra on square}, which could be used to define an `integration' operator on $\Omega^1_{big}$.
\end{rem}
There are evaluation maps $\varepsilon_s:\Omega_{big}^1\to\C$, given by sending $t\to s$ and every other generator to $0$. Since $\C\to\Omega_{big}^{1}$ is a weak equivalence, so is $A\to\Omega_{big}^1\otimes A$ for any $A\in\cbba$. Further, for the map $(\varepsilon_0,\varepsilon_1):\Omega_{big}^1\otimes A\to A\times A$ is always a fibration (i.e. surjective): A preimage for $(a,b)\in A\times A$ is given by $(1-t)a+tb$. Thus, we have shown:

\begin{lem}
	The association $A\to \Omega^1_{big}\otimes A$ is a functorial path object in $\cbba$.
\end{lem}

 We therefore say that two maps $f,g:A\to B$ in $\cbba$ are (right) homotopic if there exists a map $H:A\to \Omega_{big}^1\otimes B$ s.t. $\varepsilon_0\otimes\Id=f$ and $\varepsilon_1\otimes\Id=g$. We denote the relation of homotopy by $f\sim g$.
 Since all objects in $\cbba$ are fibrant, homotopy is an equivalence relation (see e.g. \cite[Prop. 1.2.5]{Hov99}). If $A$ is cofibrant, there is an identification 
 \[
 [A,B]\cong \Hom_{\cbba}(A,B)/\sim,
 \]
 where $\sim$ denotes the relation of homotopy and we write $[A,B]$ for the set of maps in $\Ho(\cbba)$. If we need to distinguish $[A,B]$ from the set of morphisms in ${\Ho(\bico)}$ between the underlying bicomplexes, we will write $[A,B]_{\cbba}$ and $[A,B]_{\bico}=[UA,UB]_{\bico}$.
 
 In $\cbba_0$, a functorial path object is given by $A\mapsto \Omega_{big}^1\tilde{\otimes} A$, where 
 \[
 \Omega_{big}^1\tilde{\otimes} A:=\C\oplus \Omega_{big}^1\otimes A^+
 \]
 with augmentation given by projection to the first summand. The corresponding notion of homotopy in $\cbba_0$ is given by a map $H:A\to \Omega_{big}^1\tilde{\otimes} B$. Further, we may consider $\Omega^1_{big}$ as an object of $\R\cbba$, by equipping it with the antilinear involution $\sigma$ mapping $\sigma(t)=t$, $\sigma(\del t)=\delbar t$ and $\sigma(\del\delbar t)=-\del\delbar t$. The previous discussion applies, mutatis mutandis, to $\cbba_0$, $\R\cbba$, $\R\cbba_0$ and their bounded below counterparts.

 Because of the adjunctions \eqref{eqn: adjunction bico cbba} and \eqref{eqn: adjunction bico cbba augmented}, and the definition of the model structure on $\cbba$, resp. $\cbba_0$ as a transferred model structure, we obtain induced derived functors \cite[1.3.7]{Hov99} 
 \begin{align*}
 	\Ho(\cbba)\to \Ho(\bico)\\
 	\Ho(\cbba_0)\to \Ho(\bico)
 \end{align*}
We now look at these more closely. The proof of the following two Lemmas is similar to the singly graded case:
 \begin{lem}\label{lem: ind map on homotopy classes}
	For two homotopic maps $f\sim g:A\to B$ in $\cbba$ (resp. $\cbba_0$), the underlying maps $f,g:UA\to UB$ (resp. $f^+,g^+: A^+\to B^+$) are homotopic in $\bico$. In particular, $H(f)=H(g)$ for any cohomological functor $H$.
\end{lem}
\begin{proof}
Consider a homotopy $H:A\to \Omega_{big}^1\otimes B$ between $f$ and $g$. The decomposition $\Omega_{big}^1=\C\oplus (\Omega_{big}^1)^+$ induces a decomposition $\Omega_{big}^1\otimes B=B\oplus (\Omega_{big}^1)^+\otimes B$ and since $\varepsilon_0=\varepsilon_1$ on the first summand, $f-g$ factors through the second summand. However, $(\Omega_{big}^1)^+$ is contractible (a direct sum of squares) and hence so is $(\Omega_{big}^1)^+\otimes B$. The argument for the augmented case is analogous.
\end{proof}

%
%

For any augmented cbba $(A,\varepsilon)$, denote by $QA:=Q(A,\varepsilon):=A^+/A^+A^+$ the space of indecomposables. With the induced differentials, $(QA,\del,\delbar)$ is a bicomplex. Further:
\begin{lem}\label{lem: indecomposables}
	The assignment $A\mapsto QA$ defines a functor $\cbba_0\to \bico$ which sends augmentedly homotopic maps of augmented cbba's to homotopic maps of bicomplexes.
\end{lem}
\begin{proof}
Let $H:A\to \Omega^1_{big}\tilde{\otimes} B$ an augmented homotopy between two maps $f,g:A\to B$. Note that 
\[
Q(\Omega^1_{big}\tilde{\otimes} B)=\Omega^1_{big}\otimes QB
\]
and hence we may argue as in the proof of \Cref{lem: ind map on homotopy classes}.
\end{proof}

Denoting by $C$ the cofibrant replacement functor in $\cbba_0$, we define:

\begin{definition}\label{def: homotopy bicomplex}
	The \textbf{homotopy bicomplex} of an augmented cbba $(A,\varepsilon)$ is given by $\pi(A,\varepsilon):=(QCA,\del,\delbar)$. 
\end{definition}

By \Cref{lem: indecomposables}, the homotopy bicomplex defines a functor
\[\pi: \Ho(\cbba_0)\to \Ho(\bico)\]
resp, taking real structures into account,
\[
\pi: \Ho(\R\cbba_0)\to \Ho(\R\bico).
\]
We may postcompose $\pi$ with any cohomological functor and obtain a (co-)homotopy version, e.g.
 \begin{align*}
 	\pi_\delbar^{p,q}(A,\varepsilon)&:=H_{\delbar}^{p,q}(\pi(A,\varepsilon))\\
 	\pi_\del^{p,q}(A,\varepsilon)&:=H_{\del}^{p,q}(\pi(A,\varepsilon))\\
 	\pi_{dR}^k(A,\varepsilon)&:=H_{dR}^{k}(\pi(A,\varepsilon))\\
 	\pi_{BC}^{p,q}(A,\varepsilon)&:=H_{\delbar}^{p,q}(\pi(A,\varepsilon))\\
 	\pi_A^{p,q}(A,\varepsilon)&:=H_A^{p,q}(\pi(A,\varepsilon))
 \end{align*}
 and, generalizing the last two, $\pi_{S_{p,q}}^k(A,\varepsilon):=H_{S_{p,q}}^k(\pi(A,\varepsilon))$ and $\pi_{ABC}(A,\varepsilon):=H_{ABC}(\pi(A,\varepsilon))$, $\pi_{BCA}(A,\varepsilon):=H_{BCA}(\pi(A,\varepsilon))$.
 If $A=(A_X,\varepsilon_x)$ for some pointed complex manifold $(X,x)$, we write $\pi(X,x):=\pi(A_X,\varepsilon_x)$, $\pi_{\delbar}^{p,q}(X,x):=\pi_\delbar^{p,q}(A_X,\varepsilon_x)$, etc.
 
 \begin{rem}
 	As for any bicomplex, by \Cref{thm:decomposition} the isomorphism type of $\pi(A,\varepsilon)$ in $\Ho(\bico)$ is determined by the multiplicities of all zigzags, or also by the minimal bicomplex $\pi_{ABC}(A,\varepsilon)$.
 \end{rem}  
 For any cohomological functor $H$ to some $\C$-linear category $C$ with an antilinear involution $\tau$, which is equivariant with respect to that involution $H(\bar A)=\tau H$, one obtains an induced functor on the categories of fixed points. In particular, one has an induced real structure on $\pi_{BC}^{\Cdot,\Cdot}(X,x),\pi_A^{\Cdot,\Cdot}(X,x)$ etc.

 We now make the relation to usual homotopy precise: 
 \begin{prop}\label{prop: homotopy groups via cbba vs actual homotopy groups}
 	For any connected complex manifold $X$ which has a nilpotent, finite-type underlying topological space, whenever $\pi_k(X,x)$ is abelian, there is an identification
 	\[
 	\pi_{dR}^k(X,x)^\vee\cong \pi_k(X,x)\otimes\C,
 	\]
 	compatible with the real structures on both sides.
 \end{prop}
The commutativity assumption is to avoid technicalities: For a general nilpotent space, one should replace $\pi_1(X,x)\otimes\C$ by its complexified Mal'cev Lie-algebra, see \cite{BG76}.
\begin{proof}
	The fundamental theorem of rational homotopy theory yields a natural isomorphism 
	\[
	\pi_k(X,x)\otimes\Q\cong [A_{PL}(X),A_{PL}(S^k)]_{\cdga^{\geq 0}_{\Q,0}},
	\]
	where $\cdga^{\geq0}_{\Q,0}$ denotes the category of augmented rational cdga's concentrated in non-negative degrees with the transfered model structure from bounded-below cochain complexes and $A_{PL}$ the piecewise-linear forms. As noted in \cite{LM15}, if we denote $\cdga_{\Q,0}$ the model category structure of augmented rational $\Z$-graded cdga's without degree restrictions, we have an identification \[[A_{PL}(X),A_{PL}(S^k)]_{\cdga_{\Q,0}^{\geq 0}}\cong [A_{PL}(X),A_{PL}(S^k)]_{\cdga_{\Q,0}},\]
	 because $A_{PL}(X)$ is cohomologically connected (and so admits a minimal model, which is cofibrant in both categories).
	Furthermore, $S^k$ is formal and the de Rham theorem gives an chain of quasi isomorphisms connecting $A_{PL}(X)\otimes\C$ with $A_X$. We thus have: 
	\begin{align*}
		\pi_k(X,x)\otimes\C &\cong [A_X,H^\Cdot(S^k,\C)]_{\cdga_{\C,0}}\\
							&\cong [CA_X,H^\Cdot(S^k,\C)]_{\cdga_{\C,0}}\\
							&\cong \pi^k_{dR}(X,x)^\vee
	\end{align*} 
where $CA_X$ denotes the cofibrant replacement in $\cbba_0$, which, being a composition of Hirsch extensions, is also cofibrant in $\cdga_{\C,0}$. The last equality follows from \cite[6.12 and 6.16]{BG76}.\footnote{The proof there is for non-negatively graded cdga's, but the argument is the same when considering all cdga's.}
\end{proof}
 Note that as for any bicomplex, there are the row and column filtrations $F,\bar{F}$ on $\pi(A)$. These induce spectral sequences and filtrations on the total cohomology, which we still denote by $F,\bar{F}$.
 
 \begin{cor}\label{cor: Filtrations and specsec}
 	For any complex manifold $X$, there are two conjugate filtrations on $\pi_{dR}^k(X)$, which we call the Hodge filtrations. Furthermore, there are two conjugate spectral sequences with the first page $\pi_{\delbar}(X)$, resp. $\pi_{\del}(X)$.
 \end{cor}

 \begin{rem}[Convergence of the spectral sequences]
	The spectral sequences starting with $\pi_{\delbar}(X,x)$ and $\pi_{\del}(X,x)$ converge to $\pi_{dR}(X,x)$ with its column, resp. row Hodge filtration if and only if $\pi(X,x)$ is locally bounded, see \Cref{lem: locally bounded}. Note that this does only depend on the isomorphism class of $\pi(X,x)$ in $\Ho(\bico)$. In particular, this is the case if there is a model for which all generators lie in a strictly convex region of the positive-degree half-plane. As we will see later on, this is the case for holomorphically simply connected manifolds or nilmanifolds. 
	
	If one works with the model category structure on $\cbba^{f.q.}$ instead as sketched in \Cref{rem: bounded vars}, one always obtains strong convergence, but the relation of the target to the usual homotopy groups is less clear.
\end{rem}

\begin{rem}[Diagrams]
	 One can also consider cohomological functors not landing in (bigraded) vector spaces but rather in categories of diagrams. For instance, the above constructions fit together to form a natural diagram
 	\[
 	\begin{tikzcd}
 		&\pi_{BC}^{p,q}(X,x)\ar[ld]\ar[d]\ar[rd]&\\
 		\pi_{\delbar}^{p,q}(X,x)\ar[rd]\ar[r,Rightarrow]&\pi_{dR}^{p+q}(X,x)\ar[d] &\pi_{\del}^{p,q}(X,x).\ar[ld]\ar[l, Rightarrow]\\
 		&\pi_A(X,x)&
 	\end{tikzcd}
 	\]
and one has the same maps between Schweitzer and de Rham, resp. Dolbeault cohomotopy as one for the respective cohomologies, e.g. there are maps $\pi_{S_{p,q}}^k(X)\to \pi_{dR}^{k+1}(X)$ for $k\geq p+q-1$ etc.
\end{rem}

 We now give an interpretation of various complex homotopy groups as spaces of maps in $\Ho(\cbba_0)$. This is a bigraded analogue of the aforementioned \cite[6.12 and 6.16]{BG76} used at the end of the proof of \Cref{prop: homotopy groups via cbba vs actual homotopy groups}. For any bicomplex $C$, define $S(C):=\Lambda(C)/\Lambda(C)^+\Lambda(C)^+$, i.e. make every product of two non-units trivial. For indecomposable, non-contractible $C$, the $S(C)$ may be considered as naive bigraded analogues of spheres. Then we have:

\begin{lem}
	For any augmented cbba $A$, there is an identification \[[A,S(C)]_{\cbba_0}\cong[\pi(A),C]_{\bico}.\]
\end{lem}
\begin{proof}
Note that $\pi(S(C))=C$, so the map is induced by applying the functor $\pi$. Without loss of generality, we may assume $A$ to be cofibrant with $\pi(A)=QA$ and 
\[
[A,S(C)]_{\bico_0}=\Hom_{\cbba_0}(A,S(C))/\sim.
\]
We may also assume that $A=A/A^+A^+$, and hence $QA=A^+$. Surjectivity is then clear. For injectivity, assume we have two augmented maps $f,g: A\to S(C)$ such that $f^+\sim g^+: A^+\to C$ in $\bico$ via a linear homotopy $h:A\to C[1,1]$ s.t. $f^+-g^+=[\del,[\delbar,h]]$. Now define an augmentation preserving map $A\to \Omega_{big}^1\tilde{\otimes} S(C)$ by setting for any $a\in A^+$:
\[
H(a):=(1-t)\otimes f^+a + t\otimes g^+a+ \del t\otimes [\delbar,h]a - \delbar t\otimes [\del,h]a + \del\delbar t\otimes ha
\]
Then, sorting by terms, we have:
\begin{eqnarray*}
	\del H(a)-H(\del a)&=&\del t\otimes (g^+a-f^+a-[\del,[\delbar,h]]a)+\\
						&&\delbar t\otimes [\del,[\del,h]]a+\\
						&&\del\delbar t\otimes (-[\del,h]a+\del h a - h\del a)\\
					&=&0.
\end{eqnarray*}
and similarly for $\delbar$, so this is a map of bicomplexes and thus, since the multiplication is trivial, of cbba's. By construction $\varepsilon_0 H=f$ and $\varepsilon_1 H=g$. 
\end{proof}
\begin{rem}
	If both sides are equipped with real structures and  $h$ is imaginary, $H$ in the above proof is real, so the result remains valid in $\R\cbba_0$.
\end{rem}
 
\begin{cor}
	Let $A$ be an augmented cbba. Then there are natural identifications
	\begin{align*}
		(\pi_{\delbar}^{p,q}(A))^\vee&\cong[A,S(\linepic[p-1,q])]_{\cbba_0}\\
		(\pi_{\del}^{p,q}(A))^\vee&\cong[A,S(~\hlinepic~[p,q-1])]_{\cbba_0}\\
		(\pi_{BC}^{p,q}(A))^\vee&\cong[A,S(\revLpic[p-1,q-1])]_{\cbba_0}\\
		(\pi_{A}^{p,q}(A))^\vee&\cong [A,S(\bullet[p,q])]_{\cbba_0}
	\end{align*}
\end{cor}
 
\begin{proof}
	Let us only do one case, the others are similar:
\begin{align*}
	[A,S(\linepic[p-1,q])]_{\cbba_0}&=[\pi(A),\linepic[p-1,q]]_{\bico}\\
	&=[\linepic[-p,-q],D\pi(A)]_{\bico}\\
	&=H_\delbar^{-p,-q}(D\pi(A))\\
	&=(\pi_{\delbar}^{p,q}(A))^\vee,
\end{align*}
where we have used the adjunctions \eqref{eqn: tensor-hom adjunction stable cat} and \eqref{eqn: tensor right adjoint stable cat} and \Cref{rem: Dolbeault les}.
\end{proof}

\subsubsection{Comparison with existing complex cohomotopy theories}\label{rem: NT Morgan}

	Let us compare the resulting notions of cohomotopy groups with those arising from existing theories. To avoid an overly technical discussion, we will assume $A_X$ is cohomologically simply connected. By \Cref{thm: Mimo exists} below, it then has a simply connected real cbba model.
	
	In \cite{NT78}, the authors develop a Dolbeault cohomotopy theory by considering `differential bigraded algebra (DBA)' models $(\Lambda V^{\Cdot,\Cdot}_{NT},\delbar)\to (A_X,\delbar)$, i.e. $\Lambda V_{NT}$ is a free bigraded algebra with generators in non-negative total degree, $\delbar$ is a differential of type $(0,1)$ that turns $\Lambda V$ into a nilpotent cdga and the map is a quasi-isomorphism in $H_{\delbar}$. The cohomology of the complex of indecomposables $(V_{NT},\delbar_0)$ defines Dolbeault cohomotopy groups. Forgetting the $\del$-differential, a minimal cbba model for $A_X$ is DBA-model in the sense of \cite{NT78} and so the Dolbeault cohomotopy groups of \cite{NT78} and \Cref{cor: Filtrations and specsec} agree. 
	
	By \cite[Thm. 4.2.]{HT90}, every filtered cdga $(A,F)$ with a complete ($A\cong\varprojlim A/F^p$) and cocomplete ($\varinjlim F^{-p}\cong A$) filtration admits a filtered (again complete and cocomplete) model $\varphi_{HT}:(\Lambda V_{HT},F)\to (A,F)$ in the sense that $\varphi_{HT}$ induces an isomorphism on the first page of the spectral sequence associated with the filtration. Under the completeness assumption, this implies that $\varphi_{HT}$ is also a de Rham model. Such a model is obtained by perturbing the differential of a bigraded cdga model for the first page of the spectral sequence. It is unique up to filtered homotopy equivalence \cite[Thm. 8.1.]{HT90}. The complex of indecomposables thus carries a filtration, giving rise to a spectral sequence and filtrations on the de Rham cohomotopy. Again, a minimal cbba model for $A_X$ is a filtered model in the sense of \cite{HT90} and so the induced filtrations on the cohomotopy groups agree.

	We note that a Fr\"olicher cohomotopy sequence was already constructed in \cite{NT78} where the existence of filtered models is already claimed with a short proof in \cite[Prop. 1]{NT78}, but we were not able to follow the argument \cite{NT78} that the map from the model is filtration preserving. A second proof is given in \cite[Thm. 1]{NT78}, which contains a claim that a Dolbeault quasi-isomorphism of cbba's is also a de Rham quasi isomorphism, which is false without completeness assumptions as we illustrate in \Cref{ex: bba-model Hopf} below. At any rate, these issues seem minor and one could use the results of \cite{HT90} instead.

	For compact K\"ahler manifolds, Morgan constructs a minimal cdga $\Lambda V$ with two filtrations $F,\bar{F}$ and two homotopic maps $\rho,\rho':\Lambda V\to A_X$ such that $(\Lambda V,F,\rho)$ is a filtered model for $(A_X,F)$ and $(\Lambda V,\bar{F},\rho')$ is a filtered model for $(A_X,\bar{F})$ in the sense of \cite{HT90}. In particular, the induced filtrations on the homotopy groups are again identified with those of \Cref{cor: Filtrations and specsec}. 
	
	Finally, Morgan's theory was recast in a more functorial way in \cite{Cir15} which compares well with the present approach: Denote by $\R\cbba'_0$ the full subcategory of $\R\cbba_0$ consisting of cohomologically simply connected first quadrant cbba's which satisfy the $\partial\bar\partial$-Lemma (i.e. their underlying bicomplex is a direct sum of dots and squares). Denote by $MHD_\ast^1$ the category of augmented $1$-connected Mixed Hodge diagrams as in \cite[Def. 3.5.]{Mor78}, \cite{Cir15}. We obtain a functor $\Psi:\R\cbba'_0\to MHD^1_\ast$ as follows:	For any $(A,\sigma)$ in $\R\cbba'$, we obtain a real Mixed Hodge diagram (c.f. \cite[Def. 3.5.]{Mor78}) $\Psi(A):=((A',W),(A,W,F), \iota)$, where $A':=A^{\sigma=\Id}$, the filtration $W$ denotes in both cases the trivial filtration with only jump in degree $0$, $F$ denotes the column filtration and $\iota:A'\otimes_\R\C\cong A$ is the tautological isomorphism. We obtain an induced commutative diagram
	\[
	\begin{tikzcd}
		\Ho(\R\cbba'_0)\ar[r,"\Psi"]\ar[d,"\pi"]&\Ho(MHD_\ast^1)\ar[d,"\pi"]\\
		\Ho(\R\bico)\ar[r,"\Psi"]&\Ho(MHC),
	\end{tikzcd}
	\]
	where $MHC$ denotes the category of real Mixed Hodge complexes \cite{Del74}, $\pi$ denotes the derived functor of indecomposables constructed in \Cref{def: homotopy bicomplex}, resp. \cite{Cir15}, and we keep denoting $\Psi$ the analogous functor on bicomplexes.

	\begin{rem} Note that the models in \cite{DGMS75}, \cite{NT78} and \cite{HT90} are all bigraded, but they are not cbba models in the sense of this article: In fact, the bigrading in \cite{DGMS75} is such that $d$ is of total degree $0$ and the models in \cite{NT78} are not compatible with the real structure, nor are they generally de Rham models, as we illustrate in \Cref{ex: bba-model Hopf} below. The models in \cite{HT90} are de Rham models but they similarly do not see the real structure and the differential will generally not be of bidegree $(1,0)+(0,1)$ but have more components.
	\end{rem}

\subsection{Minimal models: uniqueness}

Our first goal will be to show that, just as in the singly graded case \cite[7.6., 7.8]{BG76}, minimal models are unique up to isomorphism. 

 \begin{thm}\label{thm: weak equiv is iso on mimos}
 	Any weak equivalence between two connected minimal cbba's is an actual isomorphism.
 \end{thm}
 \begin{proof}
 	Let $A,B$ be minimal. Since $A,B$ are both fibrant and cofibrant, a map $f:A\to B$ is a weak equivalence if and only if it is a homotopy equivalence. Thus, assuming $f$ to be a weak equivalence, we may pick a $g:B\to A$ such that $f\circ g\simeq \Id_B$ and $g\circ f\simeq \Id_A$. In particular, $\pi_{ABC}(f)$ and $\pi_{ABC}(g)$ are inverse isomorphisms. Given any presentation $A=\Lambda V_A$ as a free algebra with $V_A=V_A^{\geq 1}$, we may identify  $V_A\cong QA$. Since $\del\delbar=0$ on $QA$, by \eqref{eqn: ABC}, there is a short exact sequence
 	\[
 	0\to\widetilde{\pi}_{BC}(A)\to V_A\to\pi_A(A)\to 0.
 	\]
 	Hence, by the Five-Lemma, the induced map $Q(g\circ f):QA\to QA$ is an isomorphism, and then so are all the maps $(A^+)^n/(A^+)^{n+1}\to (A^+)^n/(A^+)^{n+1}$. Since $A$ is connected, all generators lie in positive degrees and so, we have $(A^+)^n\cap A^m=\{0\}$ for $n\gg m$. Applying the Five-Lemma repeatedly (but finitely many times in every degree) to the diagrams
 	\[
 	\begin{tikzcd}
 		0\ar[r]&(A^+)^{n}/(A^+)^{n+1}\ar[r]\ar[d]& A^+/(A^+)^{n+1}\ar[r]\ar[d]&A^+/(A^+)^n\ar[r]\ar[d]&0\\
 		0\ar[r]&(A^+)^n/(A^+)^{n+1}\ar[r]& A^+/(A^+)^{n+1}\ar[r]&A^+/(A^+)^n\ar[r]&0
 	\end{tikzcd}
 \] 	
 we conclude that $f\circ g: A^+\cong A^+$ is an isomorphism, hence also $f\circ g: A\cong A$ is an isomorphism, again by the Five Lemma applied to the  augmentation sequence. Arguing analogously for $f\circ g$, we find that $f$ and $g$ must be isomorphisms.
 \end{proof}

\begin{cor}[Uniqueness of connected minimal models]
Any two connected minimal models for a cbba $A$ are isomorphic (over $A$).
\end{cor}
\begin{proof}
Given $\varphi:M\to A$ and $\varphi':N\to A$ two positively generated minimal models for a cbba $A$, by the lifting Lemma \ref{cor: nilpotent implies cofibrant}, we obtain a map $f:M\to N$ s.t. $\varphi'\circ f\sim\varphi$. Since $\varphi,\varphi'$ are weak equivalences, so is $f$ and we conclude by \Cref{thm: weak equiv is iso on mimos}.
\end{proof}
\subsection{Technical preparations}
The following Lemmas will be the main technical ingredients in the proof of the existence of minimal models for simply connected cbba's. The key point is that we need to `kill' redundant classes in Bott-Chern cohomology while controlling the rest of the cohomology and keeping the algebra minimal. For this, we have to investigate more closely than before the kind of extensions that remove the Bott-Chern classes.

For any cbba $M$ and $c\in M$ a pure bidegree element s.t. $\del c=\delbar c=0$, denote by $M_c:=M\otimes_c\Lambda(\square)$ the pushout of the following diagram:
\[
\begin{tikzcd}
	\Lambda(\bullet[|c|])\ar[r]\ar[d]&\Lambda (\square)\\
	M&
\end{tikzcd}
\]
where the horizontal map is induced by inclusion into the top corner.

\begin{lem}[Bott-Chern-classes squared away II]\label{lem: squared away}
	Let $M=\Lambda V$ be a cbba which as an algebra is freely generated in degrees $2\leq j\leq k+1$ s.t. $(dV)^{k+2}=(\del\delbar V)^{k+2}=0$. Let $c\in M$ be a pure bidegree element s.t. $\del c=\delbar c=0$. Then, the cohomology of $M_c$ can be expressed in terms of the inclusion $i:M\hookrightarrow M_c$, as follows:
	
	\begin{enumerate}
		\item\label{it: |c|=k [c]A neq 0} If $|c|=k$ and $[c]_A\neq 0$, then $i$ induces an isomorphism
		\[
		H_A^{<k}(M)\cong H_A^{<k}(M_c),\qquad H_{BC}^{<k}(M)\cong H_{BC}^{< k}(M_c)
		\]
		and a short exact sequence
		\[
		0\longrightarrow \langle [c]\rangle\longrightarrow H_{BC}^k(M)\overset{i_*}{\longrightarrow} H_{BC}^k(M_c)\longrightarrow 0
		\]
		\item\label{it: |c|=k indec} If $c$ is as in \eqref{it: |c|=k [c]A neq 0} and indecomposable, $i$ further induces  an isomorphism
		\[
		H_{BC}^{k+1}(M)\cong H_{BC}^{k+1}(M_c)
		\]
		and a short exact sequence
		\[
		0\longrightarrow \langle [c]_A\rangle\longrightarrow H_{A}^k(M)\longrightarrow H_{A}^k(M_c)\longrightarrow 0.
		\]
		\item\label{it: |c|=k+1} If $0\neq [c]\in \HBCred^{k+1}(M)$, $i$ induces isomorphisms
\[
				H_A^{\leq k}(M)\cong H_A^{\leq k}(M_c)
\]
and
\[
H_{BC}^{< k}(M)\cong H_{BC}^{< k}(M_c),\quad \HBCred^{k}(M)\cong \HBCred^{k}(M_c)
\]
		and short exact sequences
		\begin{align*}
			&0\longrightarrow H_{BC}^k(M)\longrightarrow H_{BC}^k(M_c)\longrightarrow R\longrightarrow 0\\
&0\longrightarrow \langle[c]\rangle\longrightarrow H_{BC}^{k+1}(M)\longrightarrow H_{BC}^{k+1}(M_c)\longrightarrow 0,
		\end{align*}
		where $R$ is a vector space of dimension at most $2$. Any subspace of $H_{BC}^k(M_c)$ that maps isomorphically to $R$ is necessarily indecomposable.
	\end{enumerate}
\end{lem}
Recall that an element of an augmented cbba $a\in A^+$ is called indecomposable if its projection onto the indecomposables is nonzero: $0\neq \pr(a)\in QA=A^+/A^+A^+$. Analogously, we say a subspace $W\subseteq A^+$ is indecomposable if its projection onto $QA$ is injective.

For the proof of \Cref{lem: squared away}, it will be convenient to consider the following truncations of a bicomplex:

\begin{definition}
	For any bicomplex $A$ and $k\in \Z$, we define truncated subcomplexes $\tau_k A\subseteq A$ by
	\[
	(\tau_k A)^j:=\begin{cases}
		0&\text{if }j>k+1\\
		(\ker\del\cap\ker\delbar)^{k+1} &\text{if }j=k+1\\
		(\ker\del\delbar)^{k}&\text{if }j=k\\
		A^j&\text{if }j\leq k-1.
	\end{cases}
	\]
\end{definition}
One checks that
\[
H_{BC}^j(\tau_kA)=\begin{cases}
	0&\text{if }j>k+1\\
	H_{BC}^j(A) &\text{if }j\leq k+1,
\end{cases}\quad H_A^j(\tau_kA)=\begin{cases}
	0&\text{if }j>k+1\\
	H_{dot}^{k+1}(A) &\text{if }j=k+1\\
	H_A^j(A)&\text{if }j\leq k.
\end{cases}
\]

\begin{proof}[Proof of \Cref{lem: squared away}]
	In all cases, let us denote by $p$ a generator for the bottom left space of the generating square in $\Lambda(\square)$ s.t. $\del\delbar p=c$ in $M_c$.

	For \eqref{it: |c|=k [c]A neq 0}, we consider the truncation $\tau_{k-1} M_c$. It is described as follows:
	\[
	(\tau_{k-1} M_c)^j=\begin{cases} 
		0&\text{if }j>k\\
		(\tau_{k-1} M)^{k}&\text{if }j=k\\
		(\tau_{k-1} M)^{k-1}\oplus \langle\del p,\delbar p\rangle&\text{if }j=k-1\\
		M^{k-2}\oplus \langle p\rangle &\text{if }j=k-2\\
		M^j&\text{if }j\leq k-3. 
	\end{cases}
	\]
	I.e. $\tau_{k-1}M_c=\tau_{k-1}M\oplus_c\square$ and the result follows from \Cref{lem: ABC cohomology of squared away BC-class}.

For \eqref{it: |c|=k indec}, we consider the filtration $\tau_k M_c$. It is described as follows
	\[
	(\tau_k M_c)^j=\begin{cases} 
		0&\text{if }j>k+1\\
		(\ker\del\cap\ker\delbar)\cap(M^{k+1}\oplus (V^2\otimes\langle \del p,\delbar p\rangle)\oplus (V^3\otimes\langle p\rangle))&\text{if }j=k+1\\
		(\ker\del\delbar)\cap (M^{k}\oplus (V^2\otimes\langle p\rangle)) &\text{if }j=k\\
		M^{k-1}\oplus \langle\del p,\delbar p\rangle &\text{if }j=k-1\\
		M^{k-2}\oplus \langle p\rangle &\text{if }j=k-2\\
		M^j&\text{if }j\leq k-3.
	\end{cases}
\]
Let us describe the first two expressions in terms of $\tau_k M$. Consider an element $e\in (\tau_kM_c)^{k+1}$ and write it as
\[
e=m+v'\otimes\del p+v''\otimes \delbar p + v\otimes p,
\]
with $v\in V^3$, $v',v''\in V^2$ and $m\in M$. Then 
\[
0=\del e=(\del m \pm v''\wedge c) +  (\del v'\mp v)\otimes \del p + \del v''\otimes \delbar p + \del v\otimes p
\]
and so in particular $\del v'\mp v=\del m \pm v''c=0$. By the same equation for $\delbar$, we obtain $\delbar v''\mp v=\delbar m\pm v'c=0$. Since no generator has differential in degree $k+2$, any exact element in degree $k+2$ has to be a sum of products of at least three generators. Since $c$ is indecomposable, this implies $v'c=v''c=0$ and so $v'=v''=0$. But then also $v=0$. We have shown $(\tau_k M_c)^{k+1}=(\tau_k M)^{k+1}$. In degree $k$, we argue similarly: For any $v\in V^2$, we have
	\begin{equation}\label{eqn: ddbar(vp)}
		\del\delbar (v\otimes p)=\del\delbar v\otimes p + \del v\otimes\delbar p - \delbar v\otimes\del p + v\wedge c.
	\end{equation}
	For this to lie in $M$, we necessarily have $\del v=\delbar v=0$ and then $\del\delbar (v\otimes p)=vc$ and arguing as before we see that $vc$ cannot be $\del\delbar$-exact in $M$. Thus $(\tau_k M_c)^k=(\tau_k M)^k$. In summary, we have shown $\tau_k M_c=M\oplus_c \square$ and so the statement about the cohomology computation follows again from \Cref{lem: ABC cohomology of squared away BC-class}. 
	
	Finally, for \eqref{it: |c|=k+1}, we consider also $\tau_k M_c$. This time, it is described as follows:
	\[
	(\tau_k M_c)^j=\begin{cases} 
		0&\text{if }j>k+1\\
		(\ker\del\cap\ker\delbar)\cap (M^{k+1}\oplus (V\otimes \langle p\rangle)) &\text{if }j=k+1\\
		(\tau_kM)^{k-1}\oplus\langle\del p, \delbar p\rangle &\text{if }j=k\\
		M^{k-1}\oplus\langle p\rangle &\text{if }j=k-1\\
		M^j&\text{if }j\leq k-1.
	\end{cases}
	\]
	and since $\del(m+v\otimes p)=\del m+\del v\otimes p\pm v\otimes\del p$, the degree $k+1$-part is identified with $\ker\del\cap\ker\delbar\cap M^{k+1}$. Thus, again $\tau_kM_c=M\oplus_c\square$ and we can apply \Cref{lem: ABC cohomology of squared away BC-class}.
	
For the additional statement about indecomposable classes we note that any representative of a class in $H_{BC}^k(M_c)$ that does not already lie in $H_{BC}^k(M)$ must, when written as an expression in products of generators, have a nonzero scalar multiple $\del p$ or $\delbar p$ as a summand and is hence indecomposable.
\end{proof}

In \eqref{it: |c|=k+1} of the previous Lemma, we restricted to the case of $[c]\neq 0$. However, if $[c]=0\in H_{BC}^{k+1}(M)$, then $M_c\cong M\otimes \Lambda L$ and we can even compute the entire cohomology from \Cref{thm: algebraic ABC Kuenneth}. In particular, we again have a short exact sequence
\[
0\longrightarrow H_{BC}^k(M)\longrightarrow H_{BC}^k(M_c)\longrightarrow R\longrightarrow 0,
\]
where $R$ has dimension exactly two in this case and again consists of indecomposable classes. 

Regardless of whether $[c]=0$ or not, during the construction of the minimal model, we may wish to get rid of the occuring classes in $R$, too. However, since these classes are indecomposable this would break minimality. A naive approach is to quotient out the ideal they generate. The next Lemmas show that this is indeed a viable solution.

\begin{lem}[Quotienting vs. squaring away]\label{lem: quot vs hom-quot for indec}
	Consider a cbba $M=(\Lambda V,\del,\delbar)$ which is free as a bigraded algebra. Let $c\in V$ an generator of pure bidegree which is $\del$- and $\delbar$-closed and $M_c=M(p,\del p,\delbar p \mid  \del\delbar p=c)$  as above. Then the projection map to the quotient by the ideal generated by $c$
	\[
	\pr: M_c\longrightarrow M/(c)
	\]
	is a weak equivalence.
\end{lem}

\begin{proof}
	We have to show that $\cI=\ker \pr$ is acyclic. Pick some splitting $V=V'\oplus \langle c\rangle$ and define a bigraded algebra (without specified differentials) $M':=\Lambda V'$. Then, as bigraded algebras, there is a canonical identification $M_c=M'\otimes \Lambda (p,\del p,\delbar p, c)$.
	We distinguish two cases:
	
	\textbf{Case 1: deg(a) even.} In this case, as vector spaces,
	\[
	\cI=(M'[p,c])^{\deg_p+\deg_c\geq 1}\oplus M'[p,c]\del p\oplus M'[p,c]\delbar p\oplus M'[p,c]\del p\delbar p, 
	\]
	where the superscript in the first summand means there is at least one $p$ or one $c$ in each monomial.
	
	Let now $e\in M'-\{0\}$ be any nonzero element. For any integers $n\geq 1, m\geq 0$, we consider the following subcomplex of $\cI$:
	
	\[
	S_{n,m}^{ev}(e):=
	\begin{tikzcd}
		\langle\delbar(ep^nc^m)\rangle\ar[r]&\langle\del\delbar  (ep^nc^m)\rangle\\
		\langle ep^nc^m\rangle \ar[u]\ar[r]&\langle \del(e p^n c^m)\rangle\ar[u]
	\end{tikzcd}
	\]
	We compute:
	
	\begin{align*}
		\del (p^nc^me)=&np^{n-1}c^m\del pe + p^nc^m\del e\\
		\delbar\del(p^nc^m e)=&n(n-1)p^{n-2}c^m\delbar p\del p e-np^{n-1}c^{m+1}e \\
		&- np^{n-1}c^m\del p\delbar e+np^{n-1}\delbar pc^m\del e+p^nc^m\delbar\del e
	\end{align*}
	Now the five summands are all orthogonal with respect to the decomposition
	
	\[
	M_c=\bigoplus_{\substack{k,l\in\N\\\varepsilon\bar\varepsilon\in\{0,1\}}}M'p^kl^n(\del p)^\varepsilon(\delbar p)^{\bar\varepsilon}
	\]
	and so $\del\delbar(p^nc^me)$ is not zero, which means $S_{n,m}(e)$. Furthermore, denoting by $\cB(M')$ a vector space basis for $M'$ consisting of pure bidegree elements, we see that
	\[
	\cI=\bigoplus_{\substack{n\geq 1, m\geq 0\\
			e\in \cB(M')}} S_{n,m}^{ev}(e),
	\]
	i.e. $\cI$ is acyclic.
	
	\textbf{Case 2: deg(a) odd.} In this case, define for any $e\in M'-\{0\}$ and $n,m\geq 0$ the following subcomplex of $\cI$:
	\[
	S_{n,m}^{odd}(e):=
	\begin{tikzcd}
		\langle\delbar(e(\del p)^n(\delbar p)^m p)\rangle\ar[r]&\langle\del\delbar  (e(\del p)^n(\delbar p)^m p)\rangle\\
		\langle e(\del p)^n(\delbar p)^m p\rangle \ar[u]\ar[r]&\langle \del(e(\del p)^n(\delbar p)^m p)\rangle\ar[u]
	\end{tikzcd}
	\]
	As before, we calculate
	\begin{align*}
		\del ((\del p)^n(\delbar p)^m pe)=&m(\del p)^n(\delbar p)^{m-1}c p e + (\del p)^{n+1}(\delbar p)^me -(\del p)^n(\delbar p)^m p\del e\\
		\delbar\del((\del p)^n(\delbar p)^m pe)=&-(n+m+1)(\del p)^n(\delbar p)^{m}c e + m(\del p)^n(\delbar p)^{m-1}c p \delbar e \\
		&+(\del p)^{n+1}(\delbar p)^m \delbar e + n(\del p)^{n-1}(\delbar p)^m c p\del e \\
		&- (\del p)^n(\delbar p)^{m+1}\del e + (\del p)^n(\delbar p)^m p \delbar \del e
	\end{align*}
	Again, all summands are orthogonal with respect to the vector space decomposition
	\[
	M_c=\bigoplus_{\substack{k,l\in \N\\\varepsilon,\eta\in\{0,1\}}}M'(\del p)^k(\delbar p)^l p^\epsilon c^\eta.
	\]
	Therefore, each $S_{n,m}^{odd}(e)$ is acyclic and hence so is
	\[
	\cI=\bigoplus_{\substack{n,m\in N\\ e\in \cB(M')}} S_{n,m}^{odd}(e)
	\]
\end{proof}
\begin{ex}\label{ex: algebra on square}
	Taking $M=\Lambda \langle c\rangle$, we see that the algebra freely generated by a square, $M_c=\Lambda(p,\del p,\delbar  p,\del\delbar p)$ is contractible (as an augmented cbba). In fact,
	\[
	M_c=\C\oplus\bigoplus S_{n,m},
	\]
	where $S_{n,m}:=S_{n,m}^{ev/odd}(1)$.
\end{ex}

\subsection{Minimal models: existence}\label{sec: Bigraded Minimal Models: Existence}

\begin{thm}\label{thm: Mimo exists}
	Let $A$ be a cohomologically simply connected, augmented cbba. Then there exists a simply connected minimal model $\varphi:M\to A$. 
	Furthermore, if $H_A(A)$ and $H_{BC}(A)$ are degree-wise finite dimensional, then so is $M$ and if $A$ is in $\R\cbba$, one may take $\varphi$ in $\R\cbba$, too.
	
\end{thm}

The algorithm to produce such $M\to A$ is in principle the familiar one from cdga's: Work by induction on the degree and add generators to obtain cohomological surjectivity, then add more generators to enforce all necessarily relations. This is in principle also what happens in the proof of the fibration-cofibration replacement in \Cref{thm: model cat cbba} (but the maximal approach there produces generators in negative degrees). An important difference to the cdga case is that to enforce relations a given degree, we may have to add new generators in degree $2$ less, and this makes the verification of (degree-wise)termination of this algorithm technically more involved.

\begin{proof}

We will inductively build partial models $(M_k,\varphi_k)$, where 
\begin{enumerate}
\item $M_k$ is minimal, simply connected and generated in total degrees $2,...,k$.
\item $M_k$ is a nilpotent extension of $M_{k-1}$ arising by adding generators in degrees $k,k-1$, with $\del\delbar\equiv 0$ on the degree $k$ generators.
\item $\varphi_k:M_k\to A$ is a map of cbba's extending $\varphi_{k-1}$, such that $H_{BC}^{\leq k+1}(\varphi_k)$ is injective and $H_{A}^{\leq k-1}(\varphi_k)$ is surjective. 
\end{enumerate}

To start the induction note that for cohomologically simply connected $A$, the unit map $\C\to A$ induces an isomorphism in $H_A^{\leq 1}$ and is injective in $H_{BC}^{\leq 3}$. So from now on, we assume $k\geq 2$.

\textbf{Step 1: Surjectivity} Assume we are given $(M_k,\varphi_k)$. We will construct a nilpotent extension $M_k\subseteq \widetilde{M}_k$ which is still minimal and adds only generators in degree $k,k+1$ that have no differential in degree $k+2$ and a map $\widetilde{\varphi}:\widetilde{M}_k\to A$ extending $\varphi$ s.t. $H_{BC}^{\leq k+1}(\widetilde{\varphi}_k)$ is still injective and $H_{A}^{\leq k}(\widetilde{\varphi}_k)$ is surjective.

Consider a bigraded vector space $C\subseteq \ker\del\delbar\cap A$ which projects isomorphically onto $\coker H_A^{k}(\varphi_{k})$. Then set $M':=M_k\otimes\Lambda(L(C)[1,1])$ and define $\varphi':M'\to A$ to extend $\varphi$ via the map induced by the identity on $V$. Then the inclusion $M_k\to M'$ induces an injection in Bott-Chern and Aeppli cohomology and an isomorphism in $H_{BC}^{\leq k}$ and $H_{A}^{\leq k-1}$, whereas there is a canonical isomorphism $H_{A}^{k}(M')\cong H_A^k(M_k)\oplus V$. Thus $\varphi'$ is surjective in $H_A^{\leq k}$ and injective in $H_{BC}^{\leq k}$.  If $H_{BC}^{k+1}(\varphi')$ is injective, we are done here, and can go to the next step. 

However, $H_{BC}^{k+1}(\varphi')$ will generally not be injective, as \[H_{BC}(M')=H_{BC}^{k+1}(M_k)\oplus (\del\otimes C)\oplus(\delbar\otimes C)\]. So assume we have a bigraded vector space $\{0\}\neq R\subseteq M'^{k+1}$ (`relations') projecting isomorphically onto $\ker H_{BC}^{k+1}(\varphi')$ and define $P:=R[-1,-1]$ (`primitives'). Now form $M_R':=M'\otimes_R\Lambda(\square\otimes P)$, i.e. the pushout of the diagram
\[
\begin{tikzcd}
	\Lambda(R)\ar[r]\ar[d]&\Lambda (\square\otimes P)\\
	M',&
\end{tikzcd}
\]
where the vertical map is induced by inclusion and, writing $p_r:=r\in P=R[-1,-1]$, the horizontal map is induced by $r\mapsto \del\delbar\otimes p_r$.

This is a nilpotent extension of $M'$, and by \Cref{lem: squared away}, the inclusion $M'\subseteq M'_R$ induces isomorphisms $H_A^{\leq k}(M')\cong H_A^{\leq k}(M'_R)$ and $H_{BC}^{\leq k}(M')\cong H_{BC}^{\leq k}(M'_R)$ and a short exact sequence
\[
0\longrightarrow\ker H_{BC}^{k+1}(\varphi')\longrightarrow H_{BC}^{k+1}(M')\longrightarrow H_{BC}^{k+1}(M'_R)\longrightarrow 0.
\]
We can define have a map $\varphi_R':M_R'\to A$ extending $\varphi'$ as follows: For a bigraded basis $r_i$ of $r$, pick pure-degree elements $s_i\in A$ s.t. $\del\delbar s_i=\varphi(r_i)$ and define $\varphi_R'(1\otimes p_{r_i}):=s_i$. By definition, $H_{BC}^{\leq k+1}(\varphi'_R)$ is injective and $H_{A}^{\leq k}(\varphi_R')$ is surjective. 

On the first glance, it seems we can stop here. However, since $H_{BC}^{k+1}(\varphi)$ is injective, $R$ necessarily consists of indecomposable elements and so $M'_R$ will never be nilpotent. On the other hand, $\widetilde{M}_k:=M'/(R)$ is still a nilpotent extension of $M_k$ and minimal.

In summary, we have the following diagram:
	\[
	\begin{tikzcd}
		&&M_R'\ar[rd,"\varphi_R'"]\ar[dd,"\pr"]&\\
		M_k\ar[rru]\ar[rrd]\ar[r]&M'\ar[ur]\ar[dr]&&A,\\
		&&\widetilde{M}_k&
	\end{tikzcd}
	\]
where all the maps on the left are nilpotent extensions, hence cofibrant, and the vertical map (given by projection) is an acyclic fibration by \Cref{lem: quot vs hom-quot for indec}. Thus, by the left-lifting property, there exists a map $s: \widetilde{M}_k\to M_R'$ s.t. $\pr\circ s=\Id$ and the restriction $s|_{M_k}$ coincides with the inclusion $M_k\subseteq M_R'$. We may then define a map $\widetilde{\varphi}_k:\widetilde{M}_k\to A$ by $\varphi_R'\circ s$. Because it extends $\varphi$ on $M_k$ and $\pr$ and hence $s$ are weak equivalences, it is surjective in $H_{A}^{\leq k}$ and injective in $H_{BC}^{\leq k+1}$.

\textbf{Step 2: Injectivity} To avoid overloading notation, we now assume that \[(M_k,\varphi_k)=(\widetilde{M}_k,\widetilde\varphi_k)\]from the previous step and forget about all the intermediate spaces used there. 

We want to construct $(M_{k+1},\varphi_{k+1})$, i.e. we need to make the map in $H_{BC}^{k+2}$ injective. Let $K\subseteq \ker\del\cap\ker\delbar\cap (M_k)^{k+2}$ be a bigraded vector space which maps isomorphically onto $\ker H_{BC}^{k+2}(\varphi_k)$ and let $Q:=K[-1,-1]$, where we write $q_k=k\in K[-1,-1]$. Then denote by $M'$ the pushout of the diagram
\[
\begin{tikzcd}
	\Lambda(K)\ar[r]\ar[d]&\Lambda (\square\otimes Q)\\
	M_k.&
\end{tikzcd}
\]
Note that $M'$ is a nilpotent extension of $M$ and because $M_k$ has no generators in degree $k+2$, $M'$ is minimal. Further, by \Cref{lem: squared away}, the inclusion $M_k\to M'$ induces isomorphisms $H_{BC}^{\leq k}(M_k)\cong H_{BC}^{\leq k}(M')$ and $H_A^{\leq k+1}(M_k)\cong H_A^{\leq k+1}(M')$ and short exact sequences
\[
0\longrightarrow \ker H_{BC}^{k+2}(\varphi_k)\longrightarrow H_{BC}^{k+2}(M)\longrightarrow H_{BC}^{k+2}(M')\longrightarrow 0
\]
and 
\[
0\longrightarrow H_{BC}^{k+1}(M_k)\longrightarrow H_{BC}^{k+1}(M')\longrightarrow \mathcal{R}\longrightarrow 0,
\]
where $\mathcal{R}$ is a vector space of dimension at most $2\dim K$ and class in $H_{BC}^{k+1}(M')$ mapping to something nontrivial in $\mathcal{R}$ is represented by indecomposable elements. 
Define a map $\varphi':M'\to A$ extending $\varphi_k$ as follows: Choose a basis $\{k_i\}$ for $K$ and pick $l_i\in A$ of bidegrees $|k_i|-(1,1)$ s.t. $\del\delbar l_i= k_i$. Then set $\varphi'(q_{k_i}):=l_i$. Note that by construction $H_{BC}^{k+2}(\varphi')$ and $H_{BC}^{\leq k}(\varphi')$ are injective and $H_{A}^{\leq k}(\varphi')$ are surjective. However, since we (potentially) increased $H_{BC}^{k+1}$, we might have introduced a new kernel in that degree. Pick a space $R\subseteq M'^{k+1}$ which projects isomorphically onto $\mathcal{R}$. Note that it is an indecomposable space, as it must map isomorphically onto a subspace of $(\del\otimes Q)\oplus (\delbar\otimes Q)$ under the projection 
\[
R\to Q(M')^{k+1}=Q(M)^{k+1}\oplus (\del\otimes Q)\oplus (\delbar\otimes Q)\to (\del\otimes Q)\oplus (\delbar\otimes Q).\]
Now, set $M_{k+1}:=M'/(R)$, which is a nilpotent extension of $M_k$ and minimal. To compute its cohomology and define $\varphi_{k+1}$, we consider $M'_R:=M'\otimes_R\Lambda(\square\otimes P)$ where $P:=R[-1,-1]$. By \Cref{lem: quot vs hom-quot for indec}, the projection $M'_R\to M_{k+1}$ is a quasi-isomorphism. By \Cref{lem: squared away}, the inclusion $M'\to M'_R$ induces isomorphisms $H_{BC}^{\leq k}(M')\cong H_{BC}^{\leq k}(M'_R)$, $H_{BC}^{k+2}(M')\cong H_{BC}^{k+2}(M'_R)$ and $H_A^{\leq k+1}(M)\cong H_A^{\leq k+1}(M'_R)$ and a short exact sequence
\[
0\longrightarrow\langle [r]\rangle_{r\in R}\longrightarrow H_{BC}^{k+1}(M')\longrightarrow H_{BC}^{k+1}(M'_R)\longrightarrow 0
\]
and so the map $H_{BC}^{k+1}(M_k)\to H_{BC}^{k+1}(M'_R)$ is an isomorphism. Now we can define a map $\varphi'_R$ extending $\varphi'$ as the previous times, by picking $\del\delbar$-primitives $s_i$ for a basis $\{r_i\}$ for $R$ and setting $\varphi'_R(p_{r_i}):=s_i$ and by definition it will be injective in $H_{BC}^{\leq k+2}$ and surjective in $H_{BC}^{\leq k}$. Again, $M_R'$ may not be minimal and we argue as at the end of the surjectivity step to obtain from this map the required $\varphi_{k+1}:M_{k+1}\to A$ with all the required properties.
\end{proof}
\begin{rem}
	The algorithm presented in the proof is not optimal. In fact, the necessity for `re-minimalization' at the end of the surjectivity and injectivity steps indicates that we could have made more sensible choices in adding generators beforehand. In concrete examples, it will usually be clear how to make those sensible choices. Note also that if we are just interested in a connected nilpotent (not necessarily minimal) model, we can omit these steps. The resulting model will still be degree-wise finite dimensional if $H_{ABC}(A)$ was.
\end{rem}

\begin{cor}\label{Cor: independence of base-point}
	For holomorphically simply connected manifolds, the isomorphism class of the homotopy bicomplex $\pi(X,x)$ in $\Ho(\bico)$ is independent of the basepoint $x\in X$.
\end{cor}
We will therefore often write $\pi(X)$ instead of $\pi(X,x)$.

\section{Examples and applications}

\subsection{Elementary examples} A real bigraded minimal model for projective space $\CP^n$ is given by
\[
M_{\CP^n}:=\Lambda(y,z,\del z,\delbar z~|~\del\delbar z=iy^{n+1})\qquad |y|=(1,1),~|z|=(n,n),
\]
where $y,z$ are real and $\varphi_{\CP^n}:M_{\CP^n}\to A_{\CP^n}$ sends $y$ to a K\"ahler form $\omega$ and all other generators map to zero. The isomorphism class of the homotopy bicomplex is thus given by a dot in degree $(1,1)$ and an $L$ with corner in degree $(n,n)$. 

\begin{rem}
	As this example shows, even if we start with a compact K\"ahler manifold, the homotopy bicomplex will generally not satisfy the $\partial\bar\partial$-Lemma.
\end{rem}

A real bigraded model for a torus $T^n=\C^n/\Gamma$ is given by 
\[
M_{T^n}:=\Lambda(x_1,...,x_n,\bar{x}_1,...,\bar{x}_n), \qquad |x_i|=(1,0), \quad d\equiv 0,
\]
where $x_i$ and $\bar{x}_i$ are conjugate and $\varphi_{T^n}:M_{T^n}\to A_{T^n}$ sends $x_i$ to $dz_i$ where $z_i$ are holomorphic coordinates on $\C^n$.

\subsection{Complete intersections}
We will compute the first stages of a minimal model for K\"ahler manifolds with the Hodge diamond of complete intersections of dimension $n>1$. As applications we compute the Bott-Chern and Aeppli cohomotopy groups and prove that such manifolds are formal in a strong sense.

Let $(X,\omega)$ be a compact Kähler manifold with the Hodge diamond of a complete intersection of dimension $n>1$. I.e., denoting by $\delta=1$ if $n/2\in\Z$ and $0$ else:
\[
h^{p,q}(X)=\begin{cases}1&p=q\neq n/2\\
d^{p,q}&p+q=n \text{ and }p\neq q\\
d^{p,q}+\delta&p+q=n\text{ and }p=q\\
0&\text{else},
\end{cases}
\]
where $d^{p,q}$ denotes the dimension of the space of primitive harmonic $(p,q)$-forms. E.g., for $n=2,3$ the Hodge diamonds look as follows:
\[
\resizebox{3cm}{!}{$\begin{array}{ccccc}
	&&1&&\\[3mm]
	&0&&0&\\[3mm]
	d^{2,0}&&d^{1,1}+1&&d^{0,2}\\[3mm]
	&0&&0&\\[3mm]
	&&1&&
\end{array}$}\qquad 
\resizebox{4cm}{!}{$\begin{array}{ccccccccc}
	& & & &1& & & &\\[3mm]
	& & &0& &0& & &\\[3mm]
	& &0& &1& &0& &\\[3mm]
	&d^{3,0}& &d^{2,1}& &d^{1,2}& &d^{0,3}&\\[3mm]
	& &0& &1& &0& &\\[3mm]
	& & &0& &0& & &\\[3mm]
	& & & &1& & & &\\[3mm]
\end{array}$
}
\]
Since $X$ is K\"ahler, it satisfies the $\del\delbar$-Lemma, so $H_{BC}(X)\cong H_{\delbar}(X)\cong H_A(X)$ and the entire additive quasi-isomorphism type of $A_X$ is determined by the Hodge numbers. Every complete intersection of dimension $n\geq 2$ is of course an example, but there are many others, e.g. any K\"ahler surface with $b_1=0$.

Let us build a minimal model $(M,\varphi)$ for $X$, following the above algorithm. Note that it is applicable since $A_X$ is bounded and from the conditions on the Hodge numbers, we see it is simply connected. We will only do so until we reach the stage where $H_{BC}^{\leq 2n}(\varphi)$ and $H_A^{\leq 2n}(\varphi)$ are isomorphisms, since afterwards the construction of $M$ and the map $\varphi$ becomes a formal procedure.

Set $M_3=\Lambda(x)$ with $|x|=(1,1)$ and define $\varphi_3:M_3\to A_X$ by $x\mapsto \omega$. This induces injections in $H_{BC}^{\leq 2n}$ and surjections in $H_A^{<n}$, so $(M_3,\varphi_3)=(M_n,\varphi_n)$. Next, we denote by $P$ the space of primitive harmonic forms in degree $n$ and consider formal copies $Q:=P$, $\del Q:=P[1,0]$, $\delbar Q:=P[0,1]$, where we denote $q_p:=p\in Q$, $\del q_p:=p\in \del Q$, $\delbar q_p:=p\in \delbar Q$. Then define
\[
M_{n+1}:=M_n(P,Q,\del Q,\delbar Q\mid\del\delbar q_p= i x p~\forall p\in P).
\]
Define $\varphi_{n+1}$ by $\varphi_{n+1}(p)=p$ and $\varphi_{n+1}(q_p)=\varphi_{n+1}(\del q_p)=\varphi_{n+1}(\delbar q_p)=0$. Note that since the $p$ are primitive, we have $\varphi_{n+1}(xp)=\omega\wedge p=0$ so that $\varphi_{n+1}$ is indeed a map of bicomplexes.

\begin{lem}
There are canonical identifications
\[H_{BC}^{\leq 2n}(M_{n+1})\cong (\Lambda(x,P)/(xP))^{\leq 2n},\] i.e. for all $k\leq 2n$ one has:
\[
H_{BC}^{k}(M_{n+1})\cong\begin{cases}
								\langle x^j\rangle & 2n\neq j=2k\neq n\\
								P\oplus \langle \delta x^{n/2}\rangle& k=n\\
								\Lambda^2(P)\oplus\langle x^n\rangle&k=2n\\
								0&\text{else.}
							\end{cases}
\]
\end{lem}
Furthermore, the natural map $H_{BC}(M_{2n+1})\to H_A(M_{2n+1})$ is an isomorphism in degrees $<2n$ and an injection in degree $2n$.
\begin{proof}
Consider the contractible bicomplexes 
\[
S_k(Q):=
\begin{tikzcd}
\delbar Q x^k\ar[r]&Px^k\\
Qx^k\ar[r]\ar[u]&\del Qx^k\ar[u]
\end{tikzcd}
\]
and denote by $B:=\C[x]\oplus \Lambda^{\leq 2}(P)\oplus \bigoplus_k S_k(Q)$. The inclusion $B\subseteq M_{n+1}$ is an isomorphism in degrees $<2n$ in degree $2n$ has the complement $C_{2n}:=\Lambda^2(Q)\oplus Q\otimes P$, which does not contain any $\del$ or $\delbar$-closed elements and extends to a bicomplex complement $M_{n+1}=B\oplus C$ in all degrees (note $S_k(Q)$ is injective and always splits off as a direct summand). Thus $H_{BC}^{\leq 2n}(M_{2n+1})=H_{BC}^{\leq 2n+1}(B)$ and the claim follows.
\end{proof}
Next, one needs to fix injectivity in $H_{BC}^{2n}$. Pick a (bigraded, real) space of relations $R\subseteq \langle x^n\rangle\oplus \Lambda^2(P)\subseteq M_{n+1}^{2n}$, that maps isomorphically onto the kernel of $H_{BC}^{2n}(\varphi_{n+1})$. Note that by the cohomology calculation above, it is possible to pick $R$ inside that subspace and that it is a vector space complement of $\langle x^n\rangle\subseteq \langle x^n\rangle\oplus \Lambda^2(P)$. Then, define $S:=R[-1,-1]$, $\del S:=R[0,-1]$, $\delbar S:=R[-1,0]$ and, writing again $s_r:=r\in R[-1,-1]$ etc., we define
\[
M':=M_{2n+1}(S,\del S,\delbar S \mid \del\delbar s_r = ir~\forall r\in R)
\] 
and $\varphi'$ to extend $\varphi_{n+1}$ by mapping $s_r$ to a $\del\delbar$-primitive of $\varphi(ir)$ (chosen linearly in $r$). We now have achieved surjectivity in $H_A^{\leq 2n}$ and injectivity in $H_{BC}^{\leq 2n}$.

Let us pause for a moment and note what we have shown so far:

\begin{prop}\label{prop: complete intersections}
A minimal model for a compact complex Kähler manifold $X$ with the Hodge diamond of a complete intersection of dimension $n\geq 2$ is given by a nilpotent extension
\[
M=M'\otimes \Lambda (D),
\]
where $M'=\Lambda(x,P,Q,\del Q,\delbar Q,S,\del S,\delbar S\mid \hdots)$ is as constructed as above and $D$ is concentrated in degrees $\geq 2n-1$.
\end{prop}
Since $|D|\geq 2n-1$, we can read off the bigraded homotopy groups in lower degrees:
\begin{cor}
	For $X$ as above, and $k\leq 2n-2$, one has
	\[\dim\pi_{BC}^{k}(X)=\begin{cases}
								1&k=2\\
								b_n(X)-\delta&k=n\\
								2\cdot (b_n(X)-\delta)&k=n+1\\
								0&\text{else.}
							\end{cases}
	\]
	and for $k\leq 2n-2$:
	\[
	\dim \pi_{A}^k(X)=\begin{cases}
		1&k=2\\
		2\cdot (b_n(X)-\delta)& k=n\\
		\delta\cdot (b_n(X)-1)^2+(1-\delta)\cdot {b_n(X)\choose 2}&k=2n-2\\
		0&\text{else.}
	\end{cases}
	\]
\end{cor}
\begin{proof}
	We have $\dim P=b_n(X)-\delta$, $\dim\del Q=\dim\delbar Q=\dim Q=\dim P$ and $\dim \del S=\dim\delbar S=\dim R=\dim \Lambda^2(P)$ which equals $(\dim P)^2$ or $\dim P\choose 2$, depending on the parity of $n$.
\end{proof}
We now prove a formality result. Recall \cite{MS22} that a complex manifold is strongly (bigradedly) formal, if $A_X$ can be connected by a chain of weak equivalences to a cbba $H$ with $\del\equiv\delbar\equiv 0$ (one may take $H=H_{BC}(X)$). Even though every $\partial\bar\partial$-manifold is formal in the usual (de Rham) sense by \cite{DGMS75}, this is not the case for this stronger notion of formality, as it is obstructed by the presence of bigraded higher operations, which may be nontrivial on $\partial\bar\partial$-manifolds by \cite{ST22}. It is an open question whether compact K\"ahler manifolds are strongly formal in general.\footnote{This question was resolved with a negative answer in \cite{PSZ24}.}
\begin{thm}\label{thm: formality complete intersections}
Any compact complex Kähler manifold of dimension $n\geq 2$ with the Hodge diamond of a complete intersection is strongly formal.
\end{thm}
\begin{proof}
Pick a partial model $\varphi':M'\to A_X$ as in \Cref{prop: complete intersections}. We define a map $\psi':M'\to H_{BC}(X)$ as follows: $\psi(x)=[\omega]$ , $\psi'(p)=[p]$ for any $p\in P$, and $\psi'(Q)=\psi'(S)=0$. Note that $\psi'$ is indeed a cbba map and that it induces an isomorphism in $H_{BC}^{\leq 2n}$ and $H_A^{<2n}$. 

We now want to extend $\varphi$ to a full model $\varphi:M\to A_X$ and $\psi'$ to a full bigraded quasi-isomorphism $\psi:M\to H_{BC}(X)$. To do so, we need to analyze the kernel of $H_A^{2n}(\varphi')$. For this, we will split $M'$ additively into subcomplexes up to degree $2n$. We will use the following ad-hoc notation: For any subset $V\subseteq M'$, we denote the sub-bicomplex of $M'$ generated by $V$ by $\llbracket V\rrbracket$. For example, $\llbracket x^kQ\rrbracket=S_k(Q)$ and $\llbracket S\rrbracket$ is the contractible bicomplex with underlying vector space $S\oplus \del S\oplus \delbar S\oplus R$. We distinguish two cases and leave the verification that all sums are indeed direct to the reader:

\textbf{Case 1: $n\geq 3$.}
There is a decomposition
\begin{align*}
	M'=&~\bigoplus_{k=0}^{n} \langle x^k\rangle\oplus P\oplus  \bigoplus_{k=0}^{ \lfloor n/2\rfloor} S_k(Q)\oplus \llbracket S\rrbracket\oplus \llbracket \Lambda^2(Q)\rrbracket\oplus\llbracket PQ\oplus xS\rrbracket\\
	&\oplus \text{ a subcomplex generated in degrees }\geq 2n+1.
\end{align*}
The bicomplexes $S_k(Q)$, $\llbracket S\rrbracket$, $\llbracket \Lambda^2(Q)\rrbracket$ are contractible. On the other hand, $\del\delbar$ is not injective on $\llbracket PQ\oplus xS\rrbracket$: In fact, $\del\delbar(PQ)=x\Lambda^2(P)$ and $\del\delbar (xS)=xR$ and $xR+x\Lambda P=\langle x^{n+1}\rangle + x\Lambda^2(P)$ which has smaller dimension than $PQ\oplus xS$.

\textbf{Case 2: $n=2$.}
We have an additive decomposition into sub-bicomplexes
\begin{align*}
	M'=&~ \C\oplus\langle x\rangle  \oplus P\oplus S_0(Q)\oplus \llbracket S\rrbracket \\
	&\oplus \langle x^2\rangle\oplus  S_1(Q)\oplus \llbracket \Lambda^2(S)\rrbracket\oplus\llbracket \Lambda^2(Q)\rrbracket\oplus\llbracket QS\rrbracket\oplus \llbracket PQ\oplus xS\oplus PS\rrbracket\\
	&\oplus \text{ a subcomplex generated in degrees }\geq 5
\end{align*}
where one checks that $S_1(Q)\oplus \llbracket \Lambda^2(S)\rrbracket\oplus\llbracket \Lambda^2(Q)\rrbracket\oplus\llbracket QS\rrbracket$ is contractible, but $\llbracket PQ\oplus xS\oplus PS\rrbracket$ is not.

To summarize this discussion, $H_{A}^{2n}(M')$ is a subquotient of the space $\langle x^{n}\rangle \oplus PQ \oplus (\langle x\rangle\oplus P)^2S$ and so the kernel of $H^{2n}(\varphi)$ can be represented by elements in that space. We now show that one can do slightly better:

\textbf{Claim:} One may modify $S$ in such a way that the kernel of $H_A^{2n}(\varphi')$ has a system of representatives lying in $PQ\oplus xS$ (resp. $PQ \oplus xS\oplus PS$ ).

To prove the claim let $s_1,...,s_k$ be a basis for the vector space
\[
S_A:=\{s\in S\mid \exists a\in PQ, b\in (\langle x\rangle \oplus P)^2\text{ s.t. }\del\delbar(a+bs)=0\}
\]
and choose a complement $S=S_A\oplus S'$. Now, pick a basis $b_1,...,b_l$ for $(\langle x\rangle \oplus P)^2$ and let $\alpha_{ij}\in\C$ s.t.
\[
[\varphi'(b_is_j)]=[\alpha_{ij}s_j]\in H_A^{n,n}(X).
\]
Let $b^1,...,b^l$ be a basis for $(\langle x\rangle\oplus P)^{2n-2}$ s.t. $[b_ib^j]=[\delta_{ij}x^n]$, with $\delta_{ij}=0$ if $i\neq j$ and $\delta_{ii}=1$. Define $\tilde{s}_j:=s_j-\sum_{i=1}^l\alpha_{ij}b^j$. Then $[\varphi'(\tilde{s}_j)]=0$ and $\langle \tilde{s}_1,...,\tilde{s}_k\rangle\cap S'=\{0\}$. Set $\tilde{S}:=\langle \tilde{s}_1,...,\tilde{s}_k\rangle\oplus S'$. Since we only modified by closed elements, $\tilde{S}$ is still a space of primitives for $R$ and we have $M'=\Lambda(x,P,Q,\del Q,\delbar Q,\tilde{S},\del S,\delbar S\mid ...)$. On the other hand, if we now take any element $z=c\cdot x^{n}+a+bs\in \langle x^{n}\rangle \oplus PQ\oplus (\langle x\rangle\oplus P)^2\tilde{S}$ s.t. $\del\delbar(z)=0$, we have $[\varphi'(z)]=[c x^n]$. This finishes the proof of the claim.

From now on assume $S=\tilde{S}$ has the properties of the claim. We may apply the algorithm of \Cref{thmintro: MiMo} to complete $(M',\varphi')$ to a nilpotent model $M=M'\otimes \Lambda(D)$, where $D^{<2n-2}=0$, $dD^{2n-1}\subseteq D^{2n-1}$, $dD^{2n-1}\subseteq D^{2n}\oplus \ker\del\delbar(PQ\oplus (\langle x\rangle \oplus P)^2S)$. Thus, defining $\psi$ to extend $\psi'$ by $\psi(D)=0$ yields a cbba map which is necessarily a bigraded quasi-isomorphism as it induces surjective maps between equidimensional vector spaces.
\end{proof}

\begin{rem}
	The argument used to prove the intermediate claim is inspired by \cite[Thm. 3.1.]{FM05}. It may be possible to adapt that theorem to the cbba setting and use it to prove strong formality for larger classes of manifolds, e.g. the compact K\"ahler $n$-folds which have no cohomology in degrees below $n/2$ other than that coming from the powers of the K\"ahler class.
\end{rem}

\subsection{Non-simply connected examples} 
\subsubsection{Nilmanifolds}
A nilmanifold is a quotient $X=G/\Gamma$  of a simply connected nilpotent Lie group $G$ by an (automatically co-compact) lattice $\Gamma$. A left-invariant complex structure on $X$ (with respect to the natural $G$-action) is called nilpotent if there exists a basis $\omega_1,...,\omega_{n}$ for the space of left-$G$-invariant $(1,0)$-forms $(A_X^{l.i.})^{1,0}$ s.t. for every $i$ one has
\[
d\omega_i=\sum_{j<k<i}A_{ijk}\omega_j\wedge \omega_k + \sum_{j,k<i}B_{ijk}\omega_j\wedge\bar\omega_k
\]
for some $A_{ijk},B_{ijk}\in \C$. Examples of nilmanifolds with nilpotent complex structures are when $G$ is itself a complex Lie group or when the complex structure is abelian, meaning $d (A_X^{l.i.})^{0,1}\subseteq (A_X^{l.i.})^{1,1}$. We refer to \cite{CFGU00} for a more in-depth discussion of nilmanifolds and nilpotent complex structures on these, including a more intrinsic definition in terms of descending series.

\begin{thm}
	For a nilmanifold $X=G/\Gamma$ with nilpotent complex structure, the inclusion of left-invariant forms $A_X^{l.i.}\subseteq A_X$ is a real bigraded minimal model.
\end{thm}

\begin{proof}
	The left-invariant forms are canonically a free cbba, generated by the dual of the complexified Lie-algebra $\fg=Lie(G)$, i.e. $A_X^{l.i.}=\Lambda(\fg_\C^\vee)$. The nilpotency condition means exactly that one can write $A_X^{l.i.}$ as the union of subcbba's $A_X^{l.i.}(i):=\Lambda(\omega_1,...\omega_i,\bar\omega_1,...,\bar\omega_i)\subseteq A_X^{l.i.}$ and each inclusion $A_X^{l.i.}(i)\subseteq A_X^{l.i.}(i+1)$ is a Hirsch extension. Thus, the left invariant forms are a nilpotent cbba and, since all generators are in degree $1$, necessarily minimal. By \cite[Main Theorem]{CFGU00}, the inclusion $A_X^{l.i.}\subseteq A_X$ induces an isomorphism in Dolbeault cohomology and so by conjugation also $H_{\del}$, thus it is a pluripotential quasi-isomorphism.
\end{proof}
In particular, we see that the isomorphism type of the homotopy bicomplex does again not depend on the basepoint. Furthermore, as the left-invariant forms are the free algebra on the dual Lie-algebra, one obtains:
\begin{cor}
	For a nilmanifold $X$ with nilpotent complex structure associated to a Lie algebra $\fg$ with complex structure $J$, there is an isomorphism $\pi(X)\cong \fg_\C^\vee$ in $\Ho(\R\bico)$, where we consider the right hand side as a bicomplex concentrated in total degree $1$ with the splitting $\fg_\C^{\vee}=\fg_\C^{1,0}\oplus\fg_\C^{0,1}$ induced by $J$ and trivial differentials.
\end{cor}
This means, as one would expect, that nilmanifolds with nilpotent complex structure are `holomorphically aspherical'. As the following example shows, the situation is different if we drop the assumption on nilpotency of the complex structure:

\begin{ex}\label{ex: non-nilpotent l.i.}(Homotopy bicomplex of a nilmanifold with non-nilpotent complex structure)
	The first real dimension in which nilmanifolds with non-nilpotent left-invariant structures appear is $6$. An example is determined by the following structure equations (for either choice of sign), c.f. \cite{UV14}:
	\[
	d\omega^1=0,\quad d\omega^2=\omega^{13}+\omega^{1\bar{3}},\quad d\omega^3=\pm i(\omega^{1\bar 2}-\omega^{2\bar 1}).
	\]
	Let us denote by $X$ any compact nilmanifold associated with the corresponding Lie-algebra. By \cite[§.4.2.]{Rol09}, the inclusion \[A_X^{l.i.}=\Lambda(\omega^1,\omega^2,\omega^3,\bar\omega^1,\bar\omega^2,\bar\omega^3)\subseteq A_X\] is a bigraded weak equivalence. Note that since $A_X^{l.i.}$ is connected the isomorphism type of $\pi^{\Cdot,\Cdot}(X,x)$ does not depend on the base point. However, while $A_X^{l.i.}$ is free as a cbba, it is not nilpotent. A minimal model $(M_X,\varphi_X)$ is given as  follows: As a bigraded algebra with conjugation we have
	\[
	M_X:=\Lambda( x,\bar x,y,\bar y,z,w,\bar w,p,\del p,\delbar p),
	\]
	where the conjugation is determined by taking $p,z$ to be real and otherwise as indicated by the superscripts. The degrees of the generators are determined by
	\[
	|p|=(0,0), \quad |x|=|y|=|w|=|\del p|=(1,0),\quad |z|=(1,1).
	\]
	A differential of type $(1,0)+(0,1)$ on $M_X$ is determined by 
	\[
	dp=\del p+\delbar p,\quad dx=0,\quad dy=iz,\quad  dw=xy+x\bar y,\quad  d\del p=i(\pm (x\bar{w}-w\bar x)-z).
	\]
	The map $\varphi:M_X\to A_X$ is determined by
	\[
	\varphi_X(x)=\omega^1,\quad\varphi_X(y)=\omega^3,\quad\varphi_X(w)=\omega^3, \quad\varphi_X(z)=\pm(\omega^{1\bar{2}}-\omega^{2\bar{1}}),\quad \varphi_X(p)=0.
	\]
	Note that by definition $\varphi$ factors as
	\[\begin{tikzcd}
		M_X\ar[r,"\varphi_X"]\ar[d,"\pr",swap]&A_X^{l.i.}\\
		M_X',\ar[ru,"\varphi_X'",swap]
	\end{tikzcd}\]
where $M_X':=\Lambda(x,\bar{x},y,\bar{y},z,w,\bar{w},z)/(\pm i(x\bar{w}-w\bar{x})-iz)$. Since $\varphi'_X$ is  an isomorphism and $\pr$ is a pluripotential quasi-isomorphism by \Cref{lem: quot vs hom-quot for indec}, the map $\varphi_X$ is indeed a pluripotential quasi-isomorphism.

From $M_X$ we may read off the isomorphism type of the homotopy bicomplex, which consists of two pairs of conjugate dots in degree $(1,0)$ and $(0,1)$ one $L$ with corner in degree $(0,0)$ and one reverse $L$ with corner in degree $(1,1)$. In particular, one has $\pi_{BC}^{1,1}(X)=\pi_A^{0,0}(X)=\C$.
\end{ex}

\subsubsection{Hopf manifolds}

Denote by $X_n:=(\C^n\setminus\{0\})/\Z$, where $\Z$ acts via scaling with a real constant $\lambda$ with $|\lambda|\neq 0,1$. One has diffeomorphisms $X_n\cong S^{2n-1}\times S^1$.

Following Greg Kuperberg \cite{Kup10}, we consider the real sub-cbba of $A\subseteq A_{X_n}$ generated by the following $1$ and $2$-forms:
\[
\alpha=\frac{\bar z\cdot d z}{z\cdot\bar{z}},\qquad \bar{\alpha}=\frac{z\cdot d \bar z}{z\cdot\bar{z}}\qquad \omega=i\cdot \frac{dz\cdot d\bar{z}}{z\cdot\bar{z}},
\]
where ${z\cdot\bar{z}}=\sum_{i=1}^nz_i\bar{z_i}$ etc. We have
\[
\del\alpha=0\qquad \delbar\alpha=\alpha\bar\alpha+i\omega\qquad \del\omega=-\alpha\omega.
\]
\begin{lem}\label{lem: fin dim model Hopf}
	The inclusion  $A\subseteq A_{X_n}$ is a pluripotential quasi-isomorphism. 
\end{lem}

\begin{proof}
	As noted by Kuperberg, the bicomplex $A$ consists of the invariant forms for the natural action of the compact group $U(n)\times S^1$. As such, it is a direct summand in $A_{X_n}$ by averaging, which computes the de Rham cohomology. On the other hand, the Hodge numbers of $X_n$ are known \cite{Ise60}, \cite{Bor78} and as a consequence, the Fr\"olicher spectral sequence degenerates at the first page and so the inclusion induces an isomorphism in $H_{\delbar}$ and $H_{\del}$; since $A,A_{X_n}$ are bounded this is a pluripotential quasi-isomorphism by \Cref{prop: characterization quisos}.
\end{proof}

Since $A$ is finite dimensional, we may readily compute its cohomology (c.f. \cite{Ste18}, \cite{Kup10}). Namely, we have:

\[
H_{BC}(X_n)=\begin{cases}
	\C &k=0\\
	\C[\omega-i\alpha\bar\alpha]_{BC}&k=2\\
	\C[\alpha\omega^{n-1}]_{BC}\oplus \C[\bar\alpha\omega^{n-1}]_{BC}&k=2n-1\\
	\C[\omega^{n}]_{BC}&k=2n\\
	0&\text{else}
\end{cases}
\]
and
\[
H_A(X_n)=\begin{cases}
	\C&k=0\\
	\C[\alpha]_A\oplus\C[\bar\alpha]_A&k=1\\
	\C[\alpha\bar\alpha\omega^{n-2}]_A &k=2n-2\\
	\C[\omega^n]_{A} &k=2n\\
	0&\text{else.}
\end{cases}
\]
Note that the algebra $A$ is not free. In fact, one has the relation 
\[
\omega^n=i\alpha\bar\alpha\omega^{n-1}.
\]

\begin{prop}\label{thm: Model Hopf}
	A real bigraded minimal model for $X_n$ is given by $(M,\varphi)$, where
	\[
	M=\Lambda( x,\bar{x},y,z,\del z,\delbar z~|~\delbar x = y, \del\delbar z=iy^n),
	\]
	with
	\[
	|x|=(1,0),~|y|=(1,1),~|z|=(n-1,n-1)
	\]
	and \[\varphi(x)=\alpha,~\varphi(y)=\omega-i\alpha\bar\alpha,~\varphi(z)=0.\]
\end{prop}

\begin{proof}
	$M$ is clearly minimal and real, so by it suffices to check that $\varphi:M\to A$ is a quasi-isomorphism, which is a straightforward calculation given the explicit formulas for the cohomology above.
\end{proof}

\begin{ex}[Dolbeault models for the Hopf surface]\label{ex: bba-model Hopf}
We use the example of the Hopf surfaces $X_2$ to compare bigraded models constructed here to those used in \cite{NT78}. Recall that \cite{NT78} considers cbba models (not necessarily with real structure) which are quasi-isomorphisms only with respect to $H_{\delbar}$. As such, the models constructed in this article fit into that framework, but others are possible. For instance, consider the following cbba which is free on two one-sided infinite zigzags: 
\[
M_{NT}:=\Lambda(x^{0,1},y^{1,1},x^{1,0}, y^{2,0}, x^{2,-1}, \hdots, x^{2,1},y^{3,1},x^{3,0},y^{4,0},\hdots),
\]
where the generators live in the indicated bidegree and, writing $y^{0,2}=y^{2,2}=0$, the differential is defined by
\[
d x^{p,q}=y^{p,q+1}+y^{p+1,q}.
\]
A map of cbba's $\varphi_{NT}:M_{NT}\to A_{X_2}$ with $H_{\delbar}(\varphi_{NT})$ an isomorphism is given by sending all generators to zero except
\[\varphi_{NT}(x^{0,1})=\bar\alpha,~\varphi_{NT}(y^{1,1})=\bar\alpha\alpha-i\omega,~\varphi_{NT}(x^{1,0})=-\alpha,~\varphi_{NT}(x^{2,1})=\alpha\omega.
\]
As noted in \Cref{rem: NT Morgan}, the spectral sequence computed from this model is the same as that of constructed here (and in fact as for any filtered model). However, $(M_{NT},\varphi_{NT})$ is not a de Rham model, since $H_{dR}(M_{NT})=\C\neq H_{dR}(X_2)$.
\end{ex}

\subsection{Manifolds without connected model} We have seen already in \Cref{ex: non-nilpotent l.i.} that in general a manifold does not have to admit a connected model, even if it is cohomologically connected (e.g. compact). This turns out to be not an uncommon phenomenon as the following two examples illustrate:
\begin{ex}[Riemann surfaces of higher genus]
Let $X=\Sigma_g$ be a Riemann surface of genus $g\geq 2$ with a fixed point $x\in X$. To build a model for $(X,x)$, one first adds a formal copy of the space of holomorphic and antiholomorphic forms as generators, $M':=\Lambda(\Omega^1(X)\oplus \bar{\Omega}^1(X))$. Then the tautological map $\varphi':M'\to A_X$ induces an isomorphism in Bott-Chern and Aeppli cohomology in degree $1$ and a surjection in degree $2$. To achieve injectivity in $H_{BC}^2$, one has to remove
\[
K:=\ker\left(H^2_{BC}(M')=\Lambda^2(\Omega^1(X)\oplus\bar\Omega^1(X))\longrightarrow H_{BC}^2(X)=A^{2}_X/\del\delbar(A_X^0)\right).
\]
To kill this kernel, one is forced to add a new spaces of generators $F_K\cong\del F_K\cong\delbar F_K$ s.t. $\del\delbar F_K=K$. This means the model will not be connected and a similar argument shows that no nilpotent bigraded model which is bigradedly quasi-isomorphic can exist. This example leads to very basic analytic questions the answer to which the author currently does not know: Is $\varphi'|_K$ necessarily injective? Is $X$ strongly formal?\footnote{In \cite{PSZ24} we (later) proved that a Riemann surface of genus $g\geq 2$ is never strongly formal.}
Also note that the space of primitives for $\varphi'(K)$ is a canonically defined finite dimensional space of functions on $X$ since the kernel of $\del\delbar$ in degree $(0,0)$ consists only of the constants and we fixed a base-point.
\end{ex}

\begin{ex}[CdF]
	The same argument as above shows that any connected compact complex manifold $X$ for which \[\Lambda^2(H_{BC}^1(X))\to H_{BC}^2(X)\]fails to be injective will not have a bigradedly quasi-isomorphic connected model. Note that there are two linearly independent forms $\omega_1,\omega_2\in H_{BC}^{1,0}(X)$ with $\omega_1\wedge\omega_2=0\in H_{BC}^{2,0}(X)$ if and only if $X$ fibres over a curve of genus at least $2$. (This is a formulation of the classical Castelnuovo de Franchis theorem without the K\"ahler hypothesis. The proof is the same as in the K\"ahler setting as it uses only that $\omega_1,\omega_2$ are closed holomorphic.).
\end{ex}

The examples in this and the preceeding sections beg the question:

\begin{que}
	Under what conditions does a compact complex manifold (resp. cohomologically connected real cbba) admit a connected model in $\R\cbba$? Under what conditions is there a model with finitely many generators in each degree?
\end{que}

We note that the condition of injectivity of $\Lambda^2H_{BC}^1(X)\to H^2_{BC}(X)$ is not sufficient to ensure existence of a connected model by \Cref{ex: non-nilpotent l.i.}.

\subsection{A question on realizability}
One can look at the theory in this article from the point of view of a complex geometer. Then it gives us new tools with a topological flavour to structure and study the realm of complex manifolds. Or one can take a topologist's perspective and ask what, if anything, this newly found extra structure tells us about the underlying homotopy type. The most basic instance of this is the following
\begin{que}[Sullivan] 
Which rational cdga's are quasi-isomorphic (over $\R$) to a real cbba?
\end{que}
To avoid technicalities, let us assume we are dealing with cohomologically simply connected rational cdga's $A$ s.t. $H(A;\Q)$ is a finite dimensional vector space. Then one finds a finite CW-complex $C_A$ s.t. the piecewise-linear forms of $C_A$ are quasi-isomorphic to $A$, \cite{Su77}, \cite{GM13}. One may embed $C_A$ into some large $\R^{2N}$ and thicken it up a little so that it is exhibited as a deformation retract $N(C_A)\simeq C_A$ of a smooth open submanifold $N(C_A)\subseteq \R^{2N}$. Restricting the standard complex structure of $\R^{2N}$ to $N(C_A)$ equips it with a complex structure and so there is a chain of quasi-isomorphisms of real cdga's between $A\otimes\R$ and the cdga of differential forms on $N(C_A)$, where the complexification of the latter naturally has the structure of a real cbba. Technically, this gives a positive answer to the above question, but one which is unsatisfying for a number of reasons: The manifolds $N(C_A)$ thus constructed are non-compact and will not satisfy Serre duality, even if $A$ did. They are of higher dimension than the cohomology of $A$ suggests and their cohomology (Dolbeault or Bott-Chern) will generally be infinite-dimensional.\footnote{ Using results of Eliashberg \cite{Eli90}, one can choose a complex structure on $N(C_A)$ which is Stein and so $H_{BC}^{p,q}(N(C_A))$ is finite dimensional whenever $p,q>0$ (but generally not for $p\cdot q=0$).} 
The following version thus remains open:
\begin{que}
	Which connected rational cdga's satisfying $2n$-dimensional Poincar\'e duality that have the cdga of forms on a compact almost complex manifold in their homotopy type, are quasi-isomorphic over $\R$ to cohomologically connected real cbba's with finite-dimensional cohomology concentrated in the first quadrant satisfying Serre duality?
\end{que}
Note that the realizability by compact almost complex manifolds is well understood, and depends only on the cohomology ring c.f. \cite{Mil22}, so that the above is -- in principle -- an algebraic question.

\vspace{1ex}

\noindent Jonas Stelzig, Mathematisches Institut der LMU M\"unchen, Theresienstra{\ss}e 39, 80333 M\"unchen.
jonas.stelzig@math.lmu.de
\end{document}